\documentclass[a4paper]{amsart}
\usepackage{amssymb,amsfonts}
\usepackage{amsmath}
\usepackage{amsthm}
\usepackage[dvipsnames]{xcolor}
\usepackage{braket,comment,soul}
\usepackage[shortlabels]{enumitem}
\usepackage{hyperref}
\usepackage{pdfpages,faktor}
\usepackage{graphicx,mathabx,bbm}
\usepackage{caption}
\usepackage{subcaption}
\usepackage{mathrsfs} 
\usepackage{tikz-cd} 
\usepackage{bm}
\usepackage{enumitem}
\usepackage{mathtools}
\usepackage{pgfplots}
\usepackage{import}
\usepackage{xifthen}
\usepackage{transparent}
\usepackage[normalem]{ulem}

\pgfplotsset{compat=1.18}

\numberwithin{equation}{section}
\theoremstyle{plain}
\newtheorem{theorem}{Theorem}[section]

\newtheorem{prop}[theorem]{Proposition}
\newtheorem{lem}[theorem]{Lemma}
\newtheorem{cor}[theorem]{Corollary}

\theoremstyle{definition}
\newtheorem{defn}[theorem]{Definition}
\newtheorem{remark}[theorem]{Remark}

\theoremstyle{definition}
\newtheorem{thmx}{Theorem}

\newcommand{\R}{\mathbb{R}}

\newcommand{\C}{\mathbb{C}}
\newcommand{\Z}{\mathbb{Z}}
\newcommand{\N}{\mathbb{N}}
\newcommand{\D}{\mathbb{D}}
\newcommand{\cU}{\mathcal{U}}
\newcommand{\cP}{\mathcal{P}}
\newcommand{\cH}{\mathcal{H}}
\renewcommand{\epsilon}{\varepsilon}
\renewcommand{\phi}{\varphi}

\DeclareMathOperator{\Int}{Int}

\begin{document}

\title[Uniformization of tongues  and variation of  the chaotic set]{Uniformization of tongues in Double Standard Map family and variation of maximal chaotic sets}

\begin{abstract}
 We study hyperbolic components, also known as \emph{tongues}, in the \textit{Double Standard Map} family comprising circle maps of the form:
   \begin{align*}
    f_{a,b}(x)=\left(2x+a+\dfrac{b}{\pi} \sin(2\pi x)\right) \mod 1,\ a \in \mathbb{R}/\mathbb{Z},\ 0 \leq b \leq 1.
\end{align*}  
We prove simple connectedness of tongues by providing a dynamically natural real-analytic uniformization for each tongue. For maps in a tongue, we characterize the unique maximal subset of the circle on which $f_{a,b}$ is Devaney chaotic. We also show that the Hausdorff dimension of this maximal chaotic set varies real-analytically inside a tongue. 
\end{abstract}

\begin{author}[K.~Banerjee]{Kuntal Banerjee}
\address{Presidency University, 86/1 College Street, Kolkata - 700073, West Bengal, India}
\email{kbanerjee.maths@presiuniv.ac.in}
\thanks{K.B was supported  partly by the Department of Science and Technology (DST), Govt. of India, under the Scheme DST FIST [File No. SR/FST/MS-I/2019/41]}
\end{author}

\begin{author}[A.~Bhattacharyya]{Anubrato Bhattacharyya}
\address{Presidency University, 86/1 College Street, Kolkata - 700073, West Bengal, India}
\email{anubrato02@gmail.com }
\thanks{A.B was supported by UGC [NTA Ref. No. 201610319430], Govt. of India.}
\end{author}

\begin{author}[S.~Mukherjee]{Sabyasachi Mukherjee}
\address{School of Mathematics, Tata Institute of Fundamental Research, 1 Homi Bhabha Road, Mumbai 400005, India}
\email{sabya@math.tifr.res.in, mukherjee.sabya86@gmail.com}
\thanks{S.M. was partially supported by the Department of Atomic Energy, Government of India, under project no.12-R\&D-TFR-5.01-0500, an endowment of the Infosys Foundation, and SERB research project grant MTR/2022/000248.}
\end{author}

\date{\today}

\maketitle

\setcounter{tocdepth}{1}
\tableofcontents

\section{Introduction}

In \cite{MR}, Misiurewicz and Rodrigues studied a family of self-maps of the unit circle $\mathbb{R/Z}$ given by
   \begin{align*}
    f_{a,b}(x)=\left(2x+a+\dfrac{b}{\pi} \sin(2\pi x)\right) \mod 1; \ a \in \mathbb{R/Z},\ 0 \leq b \leq 1.
\end{align*}
The above maps, which are perturbations of the doubling map on the circle, are called \emph{Double Standard Maps} (\emph{DSM} for brevity). They can be regarded as degree two analogs of certain circle diffeomorphisms called \emph{Standard Maps}:
\begin{align*}
    A_{a,b}(x)=\left(x+a+\dfrac{b}{2\pi} \sin(2\pi x)\right) \mod 1; \ a \in \mathbb{R/Z},\ 0 \leq b \leq 1.
\end{align*}
Standard maps were introduced by Arnol'd and have subsequently been studied by several people (see \cite{Arn65,EKT95}). Notably, the Double Standard Maps exhibit features of Standard Maps as well as of expanding circle endomorphisms. 

 \begin{figure}[hbt!] 
  \captionsetup{width=0.98\linewidth}
     \centering
    \includegraphics[width=.75\linewidth]{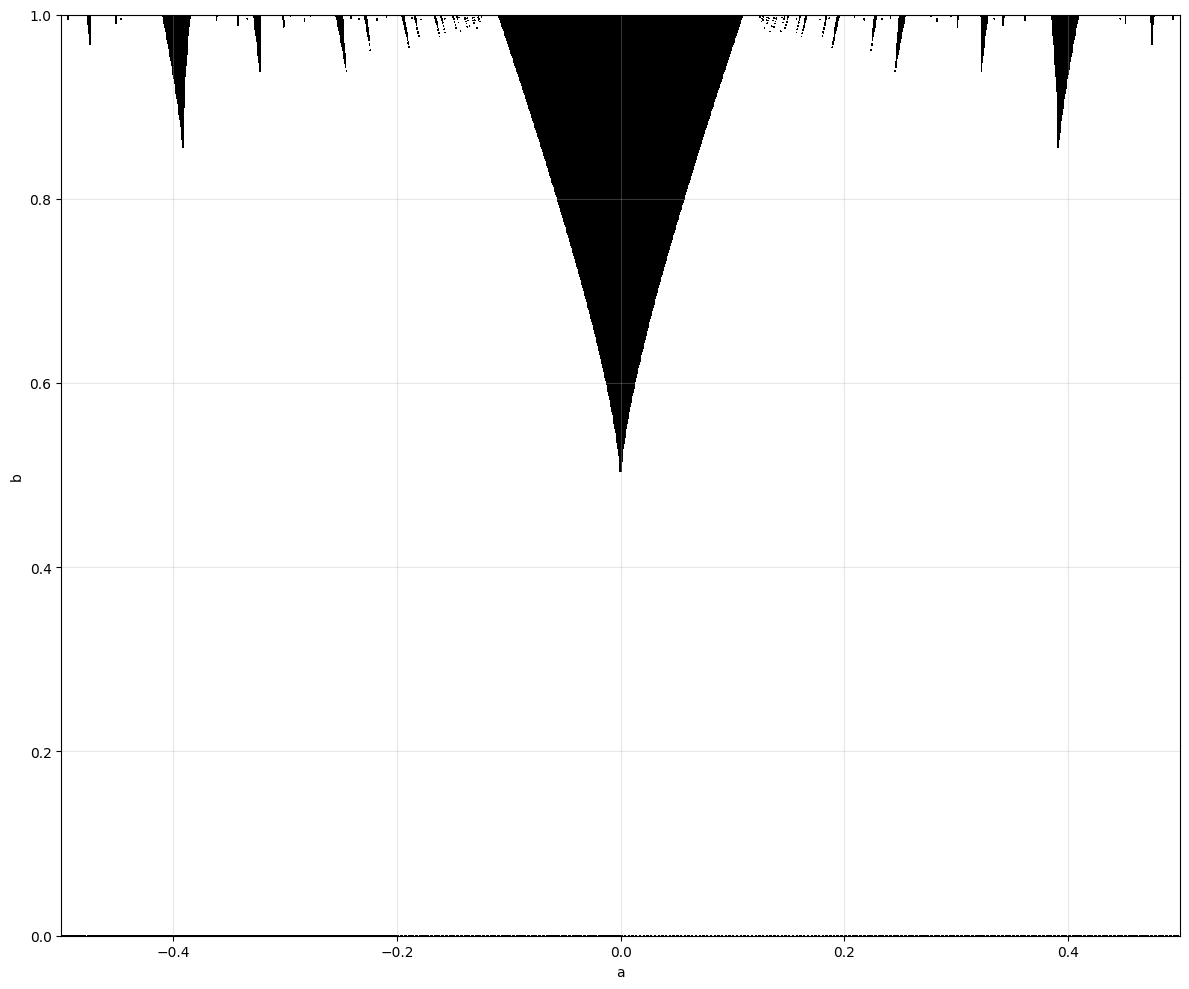}
    \caption{Depicted are tongues of period $\leq 10$ in the parameter space, $\mathbb{R/Z}\times [0,1]$, of the DSM family. Here $(a,b)  \in [-1/2,1/2] \times [0,1]$}
    \label{tongue_pic}
\end{figure}

A conspicuous feature in the parameter space 
$$
\cP:=\{(a,b): a\in \mathbb{R/Z},\ b\in[0,1]\}
$$ 
of the DSM family is the abundance of tongue-shaped \emph{hyperbolic components} sticking out of the `ceiling' $\{b=1\}$ (see Figure~\ref{tongue_pic}). Specifically, the \emph{hyperbolic locus} in the parameter space is defined as
$$
\mathcal{H}= \{ (a,b) \in\cP : f_{a,b} \text{ has an attracting cycle} \},
$$
and a \emph{tongue} is a connected component of $\cH$. Equivalently, a tongue is the collection of parameters in $\cP$ such that the associated Double Standard Maps admit an attracting cycle on $\R/\Z$ with specified combinatorics \cite{D} (see Definition~\ref{tongue_def} for a precise definition). It is worth pointing out that tongues in the DSM family have different geometry from those in the Standard Map family; indeed, no tongue in the DSM family stretches below the line $\{b=1/2\}$, while at least one tongue in the Standard Map family touches the `floor' $\{b=0\}$ (cf. \cite[Theorem~5.4]{MR}).

As is usual in real one-dimensional dynamics, a detailed analysis of the dynamics and parameter space of the DSM family is facilitated by complexifying the maps. This approach was adopted in \cite{MR,D}, where the maps $f_{a,b}:\R/\Z\to\R/\Z$ were conjugated by the natural identification 
\begin{align*}
\R/\Z\cong \mathbb{S}^1,\ x \mapsto e^{2\pi i x}
\end{align*}
giving rise to the maps
\begin{align*}
    g_{a,b}: \mathbb{S}^1 \to \mathbb{S}^1,\quad  g_{a,b}(z)=e^{2\pi i a}z^2\exp\left( bz-\dfrac{b}{z}  \right).
\end{align*}
Evidently, $g_{a,b}$ extends to a holomorphic self map of the punctured complex plane. The collection 
$$
\{g_{a,b}:\C^*\to\C^*:\ (a,b)\in\cP\}
$$
is called the \emph{complexified DSM family}.

\subsection*{Dynamically natural uniformization of tongues}
Working with the complexified maps $g_{a,b}$ makes the study of the parameter space of the DSM family amenable to techniques from holomorphic dynamics.
Our first main result proves simple connectivity of tongues in the DSM family by producing an explicit dynamically defined uniformization for the tongues.
\begin{thmx}\label{unif_thm_intro_version}
   The interior of each tongue in the complexified DSM family is simply connected. Specifically, there is a real-analytic, dynamically natural diffeomorphism from the interior of each tongue onto $\D\setminus[0,1)$. In particular, any two maps in the interior of a tongue are quasiconformally conjugate such that the quasiconformal conjugacies depend real-analytically on the parameters.
\end{thmx}
\noindent (See Theorem~\ref{parametrization_thm} and Corollary~\ref{qc_conjugate_cor} for more precise formulations.)
\smallskip 

In usual holomorphic dynamics, such a uniformization is given by the multiplier of the unique attracting cycle. However, the multiplier of an attracting cycle of $g_{a,b}$ on $\mathbb{S}^1$ is necessarily real. This forces us to bring in a new conformal conjugacy invariant, called the \emph{critical angle}. The critical angle essentially measures the angle between the two `non-real' critical points of $g_{a,b}$ in the linearizing coordinate (see Section~\ref{new_conf_inv_subsec}). More precisely, for each parameter in a tongue, the map $g_{a,b}$ has a unique attracting cycle on $\mathbb{S}^1$, and a component of the immediate basin of attraction of this cycle contains the two critical points of $g_{a,b}$. These critical points lie off $\mathbb{S}^1$, and the critical angle measures their deviation from the circle. We prove in Section~\ref{tongue_unif_subsec} that these two pieces of data: multiplier and critical angle, yield a real-analytic, bijective parametrization of a tongue.

\subsection*{The maximal chaotic set in $\mathbb{S}^1$, and its analytic motion}
For each $(a,b)\in\cP$, the restriction of $g_{a,b}$ on $\mathbb{S}^1$ (or equivalently, $f_{a,b}\vert_{\R/\Z}$) is semi-conjugate to the \emph{doubling map}
\begin{align*}
D: \mathbb{R/Z} \to \mathbb{R/Z},\ x\mapsto 2x \mod 1.
\end{align*}
In particular, the topological entropy of $g_{a,b}:\mathbb{S}^1\to\mathbb{S}^1$ is positive (at least $\ln 2$). Consequently, the dynamics of $g_{a,b}$ on $\mathbb{S}^1$ exhibits chaotic behavior. According to \cite{Mi1}, there exists a (infinite) closed, forward-invariant subset of $\mathbb{S}^1$ on which $g_{a,b}$ is 
\emph{Devaney chaotic} (see Definition~\ref{devaney_chaos_def} for the notion of Devaney chaos).

It is a straightforward consequence of standard results in one-dimensional dynamics that for any parameter $(a,b)$ outside $\overline{\cH}$, the map $g_{a,b}$ is Devaney chaotic on the entire circle $\mathbb{S}^1$ (see Proposition \ref{outside_tongue}).

On the other hand, for parameters $(a,b)$ in a tongue, we show that the $g_{a,b}$ is Devaney chaotic on $C_{a,b}$, the complement of the basin of attraction of the unique attracting cycle in $\mathbb{S}^1$ (see Proposition~\ref{chaos_inside_tongue}). In fact, $C_{a,b}$ is a Cantor set and the one-dimensional Lebesgue measure of $C_{a,b}$ is zero. Further, it is the largest subset of $\mathbb{S}^1$ on which $g_{a,b}$ is Devaney chaotic.
It is natural to ask whether the `size' of this maximal chaotic set $C_{a,b}$ varies in a regular way throughout a tongue.
Using tools from thermodynamic formalism and real-analytic motion of the chaotic set $C_{a,b}$ (which is a consequence of Theorem~\ref{unif_thm_intro_version}), we show that:
\begin{cor}\label{hd_real_anal_thm_intro_version}
The Hausdorff dimension of the maximal chaotic set $C_{a,b}$ of $g_{a,b}$ (in $\mathbb{S}^1$) depends real-analytically on the parameter $(a,b)$ throughout the interior of a~tongue. 
\end{cor}    
\noindent (See Theorem~\ref{analytic_dimension}.)    
\smallskip

For more background on the DSM family and further results, we refer the reader to \cite{BBCE,BMR23}.

\section{Background on the complexified DSM family}\label{background_sec}
We recall the Double Standard Map family (DSM for brevity)
 \begin{equation*}
 \begin{split}
     f_{a,b}:\R/\Z\to\R/\Z \hspace{4cm} \\
     f_{a,b}(x)=\left(2x+a+\dfrac{b}{\pi} \sin(2\pi x)\right),\ a \in \mathbb{R},\ b\in [0,1].
 \end{split}
\end{equation*}
 It is enough to concentrate on  $\{(a,b): a\in \mathbb{R/Z},\ b\in[0,1]\}$ as our parameter space since $f_{a+1,b}\equiv f_{a,b}$. We denote the parameter space $ \mathbb{R/Z}\times [0,1] $ as $\mathcal{P}$.  
The complexification of the DSM family is given by
\[g_{a,b}: \mathbb{C}^* \to \mathbb{C}^* \] 
\[ g_{a,b}(z)=e^{2\pi i a}z^2\exp\left( bz-\dfrac{b}{z}\right),\ (a,b)\in\cP.
 \]
We will often identify a map in the complexified DSM family with the corresponding parameter.
We set $\eta(z)=1/\overline{z}$. The map $\eta$ is the reflection in the unit circle $\mathbb{S}^1$.

\subsection{Basic properties}\label{basic_prop_subsec}
We collect some basic but important facts about the complexified DSM family in the following lemma. While we only need the relevant facts for $b\in[0,1]$, many of the statements recorded below hold for arbitrary $b\in\R$.
We refer the reader to \cite{M} for the local fixed point theory of holomorphic maps and to \cite{Kot} for details about asymptotic values and for a basic introduction to transcendental dynamics. 
\begin{lem}\label{critical}
   Let $a \in \mathbb{R}/\mathbb{Z}$ and $b \in \mathbb{R}$. Then the following are true.
\smallskip

   \begin{enumerate}[leftmargin=8mm]
    \item\label{map_symmetry} $g_{a,b}\circ \eta=\eta\circ g_{a,b}(z)$; i.e., $g_{a,b}$ is symmetric with respect to $\mathbb{S}^1$.
    \smallskip

    \item $\mathbb{S}^1$ is invariant under $g_{a,b}$. In particular, $g_{a,b}(\mathbb{S}^1) =\mathbb{S}^1$.
       \smallskip
       
       \item The map $g_{a,b}$ has critical points at $\dfrac{-1 \pm \sqrt{1-b^2}}{b}$. These are
       \begin{enumerate}
           \item two distinct $\mathbb{R}$-symmetric points on $\mathbb{S}^1$ when $|b|> 1$,
           \item \label{symmetric_critical}two distinct $\mathbb{S}^1$-symmetric points on $\mathbb{R}$ when $|b|< 1$, and
           \item a unique double critical point at $\pm 1$ for $b=\mp1$.
       \end{enumerate}
       \smallskip
       
            \item The map $g_{a,b}$ has essential singularities at $0$ and $\infty$, which are also the only asymptotic values of $g_{a,b}$. 
       \smallskip
       
       \item  For $ a \in \mathbb{R}/\mathbb{Z}, 0 \leq b < 1$, the map $g_{a,b}$ is orientation preserving restricted to $\mathbb{S}^1$ with critical points away from the unit circle.
       \smallskip
       
       \item If $g_{a,b}$ has an attracting cycle on $\mathbb{S}^1$, for $b\in[-1,1]$, then both the critical points of $g_{a,b}$ lie in the same component of the immediate basin of this attracting cycle. In particular, there is no other attracting cycle for $g_{a,b}$.
       \smallskip
       
       \item\label{koenig_sym} If $g_{a,b}$ has an attracting cycle on $\mathbb{S}^1$ of period $q$ and multiplier $\lambda$, then $\lambda \in \mathbb{R}$. Let $\kappa$ be the linearizing map defined in some neighborhood of a point $x$ on this cycle, normalized so that $\kappa(x)=0,\ \kappa'(x)=C\dfrac{i}{x}$, for some $C\in\R\setminus\ \{0\}$. Then, 
       \begin{equation}
           \kappa\circ\eta(z)=\overline{\kappa(z)}.
           \label{kappa_symmetry_eqn}
       \end{equation} 
   Conversely, if $\kappa$ is a linearizing map defined in some neighborhood of a point $x$ on this cycle satisfying $\kappa(x)=0$ and $\kappa\circ\eta(z)=\overline{\kappa(z)}$, then $\kappa'(x)=C \dfrac{i}{x}$ for some $C \in \mathbb{R} \setminus \{0\}$.

       In particular, the pre-image of any disc centered at $0$ under $\kappa$ is also symmetric with respect to $\mathbb{S}^1$.
       \smallskip
       
       \noindent (It follows that if $D_R:=\{z\in\C:\vert z\vert<R\}$ is a disc on which $\kappa^{-1}$ is well defined then $\mathcal{U}:=\kappa^{-1}(D_R)$ is symmetric with respect to $\mathbb{S}^1$). 
   \end{enumerate}
\end{lem}
\begin{proof}
Items~(1),~(2), and~(3) can be verified by elementary computations.
Item~(4) is the content of \cite[Proposition~2.3]{D}, Item~(5) is proved in \cite[Lemma~2.2]{D}, and Item~(6) follows from \cite[Proposition~2.5]{D}.
\smallskip

\noindent \underline{Proof of Item~\ref{koenig_sym}.}
The fact that $g_{a,b}$ preserves the unit circle implies that the multiplier $\lambda$ is real.
Define $\widetilde{\kappa}(z)=\overline{\kappa\circ \eta(z)}$. Since $\lambda$ is real, we have
\[
\widetilde{\kappa}\circ g_{a,b}^{\circ q}(z)=\overline{\kappa \circ \eta \circ g_{a,b}^{\circ q}(z)}= \overline{\kappa\circ g_{a,b}^{\circ q} \circ \eta(z)}=\lambda \overline{ \kappa\circ \eta(z)}= \lambda \widetilde{\kappa}(z).
\]
This means that $\widetilde{\kappa}$ is also a linearizing map. Observe that $\widetilde{\kappa}(x)=0$ since $\kappa(x)=0$. 
A direct calculation shows that
\[
\kappa \circ \eta(x+\epsilon)=\kappa\left( (1/\overline{x}) \left(1+\overline{\epsilon}/\overline{x}\right)^{-1}\right)=\kappa\left( x(1+\overline{\epsilon}x)^{-1} \right)=\kappa\Bigl( x\left(1-\overline{\epsilon}x+o(\epsilon)\right) \Bigr),
\]
where we have used the fact that $\eta(x)=1/\overline{x}=x$ as $x \in \mathbb{S}^1$.
As $\kappa(x)=0$, this means
\[
\kappa \circ \eta(x+\epsilon)= \kappa(x-\overline{\epsilon}x^2+o(\epsilon))=\kappa'(x)(-\overline{\epsilon}x^2)+o(\epsilon).
\]
Taking complex conjugate on both sides yields
\begin{align}
\notag\widetilde{\kappa}(x+\epsilon)&\ =\overline{\kappa \circ \eta(x+\epsilon)} =-\overline{\kappa'(x)x^2}\epsilon+o(\epsilon)\\
&\implies \widetilde{\kappa}'(x)=-\overline{\kappa'(x)x^2}
\label{kappa_tilde_power_series_eqn}
\end{align}

If $\kappa'(x)=C\dfrac{i}{x}$ for a non-zero real number $C$ (note that such a normalization is possible since any non-zero complex multiple of a linearizing map  is also a linearizing map), then Relation~\eqref{kappa_tilde_power_series_eqn} implies that $\widetilde{\kappa}'(x)=\kappa'(x)$. It follows that under the normalization $\kappa'(x)=C\dfrac{i}{x}$, the linearizing maps $\widetilde{\kappa}$ and $\kappa$ have the same derivative at $x$.
Since the linearizing map is unique if the derivative at $x$ is specified (cf. \cite[Theorem~8.2]{M}), we conclude that $\widetilde{\kappa}=\kappa$; i.e.,  $\kappa~\circ~\eta(z)~=~\overline{\kappa(z)}$.

Conversely, suppose that $\widetilde{\kappa}=\kappa$, implying $\widetilde{\kappa}'(x)=\kappa'(x)$. Further, let $\kappa'(x)=C' e^{i\theta}$, for some $C'>0$ and $\theta \in [0,2\pi)$. Then, Relation~\eqref{kappa_tilde_power_series_eqn} reduces to the requirement
\[
-\overline{e^{i\theta} x^2}=e^{i\theta} \implies -\overline{x}^2=e^{2i\theta}\implies e^{i\theta}=\pm i\overline{x}=\pm\dfrac{i}{x}.
\]
Thus, the equality $\widetilde{\kappa}=\kappa$ (of maps) forces the normalization $\kappa'(x)=C\dfrac{i}{x}$ for some $C \in \mathbb{R} \setminus \{0\}$.
\end{proof}

\begin{remark}\label{koenig_geometry_rem}
For an attracting periodic point $x\in\mathbb{S}^1$ of $g_{a,b}$, let us denote by $\vec{\ell}_x$ the tangent direction to $\mathbb{S}^1$ at $x$ pointing in the clockwise direction. Then, $\vec{\ell}_x$ makes an angle $\arg{\dfrac{x}{i}}$ with the positive real axis. Hence, the normalization $\kappa'(x)=C\dfrac{i}{x}$, where $C>0$, is equivalent to saying that $\kappa$ sends the tangent direction $\vec{\ell}_x$ to the positive real direction at the origin.
\end{remark}

For a map $g_{a,b}$, $\vert b\vert\leq 1$, with an attracting cycle on $\mathbb{S}^1$, we call the unique component of its basin of attraction containing the two critical points the \emph{distinguished basin component} of $g_{a,b}$. The  point of the attracting cycle in the distinguished basin component will be referred to as the \emph{distinguished attracting periodic point} of $g_{a,b}$.

Let  $f_{a,b}$ be a map in the DSM family that has an attracting periodic orbit $P$ of period $q$. Let $x$ be the distinguished periodic point of $P$. The map $f_{a,b}$  is semi-conjugate to the doubling map $D$; i.e., there exists a degree 1, monotonically non-decreasing circle map $\varphi_{a,b}$ so that $\varphi_{a,b} \circ f_{a,b}=D\circ \varphi_{a,b}$ (see \cite[Lemma~3.2]{MR}). Further, it follows from \cite[Lemma~3.2]{MR} that $\widetilde{x}=\varphi_{a,b}(x)$ is a periodic point of the doubling map with period $q$. The point $\widetilde{x}$ is called the \emph{type of the orbit $P$}.

\begin{defn}[Tongues in the complexified DSM family]\label{tongue_def}
For a periodic point $\widetilde{x}$ of the doubling map $D$ of period $q$, we define the \emph{tongue of type $\widetilde{x}$ and period $q$} as the set of parameter values $(a,b)\in \mathcal{P}=\R/\Z\times [0,1]$ for which $g_{a,b}$ has an attracting $q-$cycle on $\mathbb{S}^1$ of type $\widetilde{x}$ (and hence $f_{a,b}$ has an attracting $q-$cycle on $\R/\Z$ of type~$\widetilde{x}$). 
\end{defn}
Note that any  periodic point of period $q$ of  $D$ is of the form $\dfrac{k}{2^q-1}$ for some $ k \in \{1,..,2^q-2\}$.

\subsection{Connectedness of tongues {\`a} la Dezotti}\label{dezotti_summary_subsec}

Let us briefly explain the scheme of the proof of \cite[Theorem~1.4]{D}. Recall that $g_{a,b}$ is the complexification of the degree $2$ circle map, $f_{a,b}(x)=\left(2x+a+\dfrac{b}{\pi} \sin(2\pi x)\right) \mod 1$. First  start with a parameter $(a,b)\in \mathbb{R/Z}\times[0,1)$ which is in some tongue $T$ of period $q$. Then $f_{a,b}$ has a geometrically attracting cycle of period $q$, and hence $g_{a,b}$ has a geometrically attracting $q-$cycle on $\mathbb{S}^1$. Dezotti uses a quasiconformal deformation technique (to be outlined below) to construct a path satisfying the following properties. A comprehensive account of quasiconformal deformations of holomorphic dynamical systems can be found in \cite{BF}.
\begin{enumerate}[leftmargin=8mm]
    \item The path $\gamma: (0,1) \to T,\  t\mapsto (a_t,b_t)$ is parametrized by the multiplier $t$; i.e., the map $g_{a_t,b_t}$ has a $q-$periodic attracting cycle on $\mathbb{S}^1$ with multiplier~$t$.

    \item For each $t \in (0,1)$, the map $g_{a_t,b_t}$ is conjugate to the initial map $g_{a,b}$  via a quasiconformal homeomorphism $\Phi_t$.

    \item The quasiconformal homeomorphism $\Phi_t$, which is unique up to a rotation, varies analytically with the parameter $t$.

    \item As $t\to 0$, $(a_t,b_t)\to (a_0,1)$, where $(a_0,1)$ is the unique parameter in $T$ admitting a superattracting $q-$cycle.
\end{enumerate}

Here is summary of the  quasiconformal deformation method. As mentioned above, one begins with a map $g_{a,b}\in T$ having a geometrically attracting cycle. By Lemma~\ref{critical} (Item~\ref{koenig_sym}), the linearizing map $\kappa$ around the distinguished attracting periodic point $x$ of $g_{a,b}$ conjugates $g_{a,b}^{\circ q}$ to $w\mapsto \lambda w$, where $\lambda:=(g_{a,b}^{\circ q})'(x)$, and conjugates $\eta$ (reflection in $\mathbb{S}^1$) to the complex conjugation map (reflection in $\mathbb{R}$). For a given $t \in (0,1)$ one considers the quasiconformal map $\chi_t(z)=z\vert z\vert^\alpha$, $\alpha=\dfrac{\ln t}{\ln \lambda}-1$ (cf. Lemma~\ref{quasi_homeo}). The map $\chi_t$ conjugates the linear map $w\mapsto \lambda w$ to $w\mapsto tw$. The real-symmetric Beltrami coefficient $\mu_t:=\overline{\partial}\chi_t/\partial\chi_t$ induced by $\chi_t$ is then pulled back to an $\mathbb{S}^1-$symmetric neighborhood of $x$ via the linearizing map $\kappa$ (see \cite[\S 1.2.1]{BF} for the notion of real-symmetry of a Beltrami coefficient). The resulting Beltrami coefficient $\sigma_t$, which is symmetric with respect to $\mathbb{S}^1$, is propagated to the whole basin of attraction via the iterates of $g_{a,b}$, and is extended trivially to the rest of $\C^*$. Since $g_{a,b}$ commutes with $\eta$, it follows that the above procedure gives an $\mathbb{S}^1$-symmetric, $g_{a,b}$-invariant Beltrami coefficient $\sigma_t$ on $\C^*$ that varies analytically with $t$ (cf. Lemma~\ref{symm_linear}).

The Measurable Riemann Mapping Theorem guarantees that there exists quasiconformal homeomorphisms $\Phi_t$ fixing $0,\infty$ such that $\overline{\partial}\Phi_t/\partial\Phi_t=\sigma_t$ a.e. on $\C^*$, and $t \mapsto \Phi_t(z)$ is analytic for any fixed $z$. Moreover the map $\Phi_t$ is symmetric with respect to $\mathbb{S}^1$ because of the $\eta$-symmetry of $\sigma_t$ (cf. Lemma~\ref{symm_linear}). Conjugating $g_{a,b}$ by $\Phi_t$, one gets a holomorphic map with an attracting $q$-cycle of multiplier $t$ (where the holomorphicity of the conjugated map follows from $g_{a,b}$-invariance of $\sigma_t$ and the Weyl's lemma). That the quasiconformally deformed map is also a member of the complexified DSM family (up to affine conjugation) is the content of \cite[Proposition~3.6]{D} which is proved using techniques from complex analysis. We state here a special case of  \cite[Proposition~3.6]{D} that is sufficient for our purposes.

\begin{lem}\cite[Lemma~2.1, Proposition~3.6]{D}\label{deformation_in_sm}
    Let $(a,b) \in \mathbb{R}/\mathbb{Z} \times \mathbb{C}^* $ and let $\widetilde{g}_{a,b} : \mathbb{C}^* \to \mathbb{C}^*$ be defined by
    $$
   \widetilde{g}_{a,b}(z) = e^{2i\pi a} z^2 e^{-(bz - \overline{b}/z)}.
    $$
    Let $\varphi, \psi$ be orientation-preserving homeomorphisms of the Riemann sphere $\widehat{\C}$ fixing $0$ and $\infty$ such that both $\phi,\psi$ commute with $\eta(z)=\dfrac{1}{\overline{z}}$. 
    \begin{enumerate}[leftmargin=8mm]
        \item If $\psi \circ \widetilde{g}_{a,b} \circ \varphi : \mathbb{C}^* \to \mathbb{C}^*$ is holomorphic, then there exists $(\alpha, \beta) \in \mathbb{R}/\mathbb{Z} \times \mathbb{C}^*$ such that $\psi \circ \widetilde{g}_{a,b} \circ \varphi = \widetilde{g}_{\alpha, \beta}$, and we have $|\beta| < 1$ if and only if $|b| < 1$ and $|\beta| = 1$ if and only if $|b| = 1$.
        \item Further, there exists a unique rotation, given by $z \mapsto \dfrac{b}{\vert b\vert} z$, that conjugates $\widetilde{g}_{a,b}$ to the map $\widetilde{g}_{\widetilde{a},\widetilde{b}} = e^{2i\pi \widetilde{a}} z^2 e^{-(\widetilde{b}z - \widetilde{b}/z)}$ with $\widetilde{b}=\vert b\vert \in \mathbb{R}_+$ and $\widetilde{a}\in \R/\Z$. 
    \end{enumerate}
    \end{lem}

\begin{remark}
Note that the maps $\widetilde{g}_{a,b}$ considered by Dezotti are related to our maps $g_{a,b}$ via the relation $\widetilde{g}_{a,b}=g_{a,-b}$. Evidently, this does not change the family of maps under investigation, and hence Lemma~\ref{deformation_in_sm} remains valid for our family $\{g_{a,b}\}$ as well.
\end{remark}

Finally, Dezotti shows that the type of the distinguished attracting periodic point remains unchanged under the above quasiconformal deformation. Thus, in light of Lemma~\ref{deformation_in_sm} and the preceding construction, the map $(0,1)\ni t\mapsto (a_t,b_t)$ yields a `multiplier--parametrized' path in $T$ containing the initial parameter $(a,b)$. It is further argued in \cite[\S 4.2]{D} that as $t\to 0^+$, this path limits at the unique superattracting parameter in $T$. This shows that any geometrically attracting parameter in $T$ can be connected to the unique superattracting parameter by a path, from which connectedness of $T$ follows.

\section{Dynamically natural uniformization of tongues}\label{tongue_unif_sec}

The goal of this section is to modify Dezotti's quasiconformal deformation scheme to produce dynamically defined uniformizations of tongues in the complexified DSM family. This will show, in particular, that tongues are simply connected.

\subsection{Introducing a new conformal invariant}\label{new_conf_inv_subsec}

The quasiconformal deformation argument summarized in Section~\ref{dezotti_summary_subsec} is based on changing the multiplier of an attracting cycle via a real one-parameter family of quasiconformal homeomorphisms $\chi_t$, $t\in(0,1)$. Since the multiplier of the attracting cycle of any parameter in a tongue is real, one cannot use multipliers to uniformize a tongue (see \cite[\S 5]{NS03}, \cite[\S 7]{LLMM25} for similar situations in parameter spaces of antiholomorphic maps). In order to provide a uniformization of a tongue, we will introduce an additional \emph{conformal invariant} called the \emph{critical angle} for maps in a tongue. The desired uniformization will be achieved by deforming the multiplier and the critical angle simultaneously.

 Let $T$ be a tongue of period $q$ in the complexified DSM family, and $g_{a,b}\in T$. Let $x\in\mathbb{S}^1$ be the distinguished attracting point of $g_{a,b}$, and $\vec{\ell}_x$ be the tangent direction to $\mathbb{S}^1$ at $x$ pointing in the clockwise direction. We normalize the linearizing map $\kappa_{a,b}$ in a neighborhood of $x$ such that $\kappa_{a,b}(x)=0$ and $\kappa_{a,b}'(x)=C\dfrac{i}{x}$, for some $C>0$. By Lemma~\ref{critical} (Item~\ref{koenig_sym}), we have that $\kappa_{a,b}\circ\eta(z)=\overline{\kappa_{a,b}(z)}$. Recall from Remark~\ref{koenig_geometry_rem} that the linearizing map $\kappa_{a,b}$ sends the tangent direction $\vec{\ell}_x$ to the positive real direction at the origin.
 
 We analytically continue $\kappa_{a,b}$ to the entire distinguished basin component, and continue to call it $\kappa_{a,b}$. The extended map semi-conjugates $g_{a,b}^{\circ q}$ to the linear map $w\mapsto (g_{a,b}^{\circ q})'(x)\cdot w$. Let $c_1,c_2$ be the critical points of $g_{a,b}$ outside and inside of the unit disk, respectively. Recall that $c_2=\eta(c_1)$, and hence $\kappa_{a,b}(c_2)=\overline{\kappa_{a,b}(c_1)}$. By our normalization, the point $\kappa_{a,b}(c_1)$ (respectively, $\kappa_{a,b}(c_2)$) lies in the upper (respectively, lower) half-plane. Further, the positive real axis bisects the angle subtended by the points $\kappa_{a,b}(c_1)$ and $\kappa_{a,b}(c_2)$ at the origin.
 
\begin{defn}[Critical angle]\label{crit_angle_def}
With the setup as above, the angle between the positive real axis and the straight line segment $[0,\kappa_{a,b}(c_1)]$ is called the \emph{critical angle} of the map $g_{a,b}$. Equivalently, the critical angle of $g_{a,b}$ is
$$
\arg{\kappa_{a,b}(c_1)}=\arg{\sqrt{\dfrac{\kappa_{a,b}(c_1)}{\kappa_{a,b}(c_2)}}}\ ,
$$
where arguments are taken in $[0,2\pi)$ and the square root is chosen in the upper half--plane.
\end{defn}

\begin{remark}\label{crit_anglr_rem}
1) By definition, the critical angle lies in the interval $(0,\pi)$. (See Figure~\ref{qc_def_fig} for an illustration.) 

2) Although the Koenigs coordinate $\kappa_{a,b}$ is normalized up to multiplication by a positive real constant $C$, the quantity $\arg{\kappa_{a,b}(c_1)}$ does not depend on the choice of the constant $C>0$.

3) The equality $\arg{\kappa_{a,b}(c_1)}=\arg{\sqrt{\dfrac{\kappa_{a,b}(c_1)}{\kappa_{a,b}(c_2)}}}$ is only valid under the normalization $\kappa_{a,b}'(x)=C\dfrac{i}{x}$, where $C>0$. However, the latter expression $\dfrac{\kappa_{a,b}(c_1)}{\kappa_{a,b}(c_2)}$ remains unaltered  if the linearizing map is multiplied with a non-zero complex number, and hence this yields a canonical way of defining the critical angle.
The chosen normalization of $\kappa_{a,b}$ merely helps us streamline some of the analysis below.
\end{remark}

\subsection{A model change of coordinates in the linearized planes}

In this subsection, we will define a real-analytic family of quasiconformal homeomorphisms of $\C$ that will be used to deform the multiplier and the critical angle of maps in a tongue. Such maps will be required to conjugate a linear map $w\mapsto \lambda_0 w$ to another linear map $w\mapsto \lambda_1 w$ and change the arguments of points in a controlled way. To attain this goal, we first define a family of piecewise linear increasing homeomorphisms. For each $\nu_0,\nu_1 \in (0,\pi)$, let
\[
h^{\nu_1}_{\nu_0}:[0,\pi] \to [0,\pi]
\]
be the unique piecewise linear increasing homeomorphism with two linear branches such that
\[
h^{\nu_1}_{\nu_0}(0)=0,\ h^{\nu_1}_{\nu_0}(\nu_0)=\nu_1,\ \mathrm{and}\ h^{\nu_1}_{\nu_0}(\pi)=\pi.
\]
Concretely, the map $h^{\nu_1}_{\nu_0}$ can be expressed as
\[
h^{\nu_1}_{\nu_0}(\theta)=
\begin{cases}
    s_1 \theta,\qquad \theta\in[0,\nu_0],\\
    s_2 \theta+ s_3,\quad \theta\in[\nu_0,\nu_1], 
\end{cases}
\]
where $s_1=\dfrac{\nu_1}{\nu_0}$, $s_2=\dfrac{\pi-\nu_1}{\pi-\nu_0}$, and $s_3=\dfrac{\pi(\nu_1-\nu_0)}{\pi-\nu_0}$.
Now consider the following family of maps:
\[ \chi\equiv \chi^{\nu_0,\nu_1}_{\alpha} : \overline{\mathbb{H}} \to \overline{\mathbb{H}}
\]
\[
\begin{cases}
 \chi^{\nu_0,\nu_1}_{\alpha}(r,\theta)= (r^{1+\alpha},h^{\nu_1}_{\nu_0}(\theta)),\ r>0,\ \theta \in [0,\pi],\\
\chi^{\nu_0,\nu_1}_{\alpha}(0)=0,
\end{cases}
\]
where $\alpha>-1$, $\nu_0,\nu_1 \in (0,\pi)$, and $\mathbb{H}$ is the upper half-plane.
The map $\chi^{\nu_0,\nu_1}_{\alpha}$ can be extended to a 
homeomorphism of $\C$ via the Schwarz reflection principle.

\begin{remark}
For each fixed $\nu_0$ and $(r,\theta)$, the map $(\alpha,\nu_1) \mapsto \chi_{\alpha}^{\nu_0,\nu_1}(r,\theta)$ is real-analytic  on the domain $(-1, \infty) \times (0, \pi)$. Indeed, observe that 
$$
r^{1+\alpha} = \exp((1+\alpha)\ln r)$$ is analytic in $\alpha$. Moreover, for $\theta \in [0, \nu_0)$,
$$
h^{\nu_1}_{\nu_0}(\theta) = \nu_1\dfrac{\theta}{\nu_0} ,
$$
and for $\theta \in (\nu_0,\pi]$, 
$$
h^{\nu_1}_{\nu_0}(\theta) = \nu_1 \left( \dfrac{\pi - \theta}{\pi - \nu_0} \right) + \dfrac{\pi\theta - \pi\nu_0}{\pi - \nu_0}.
$$
For the break point $\theta = \nu_0$, both branches yield $h(\nu_0) = \nu_1$. In all cases we get a linear function of $\nu_1$, which is real-analytic. Since here the first coordinate of $\chi_{\alpha}^{\nu_0,\nu_1}(r,\theta)$ is only a function of $\alpha$  and the second coordinate is only a function of   $\nu_1$, the map remains jointly real-analytic in the parameters $\alpha, \nu_1$ across the entire domain. We shall see in the next lemma that $\chi_{\alpha}^{\nu_0,\nu_1}(r,\theta)$ gives rise to a Beltrami coefficient that also depends real-analytically on parameters. This real-analytic parameter dependence will play a key role in the proof of our main theorem.      
\end{remark}

  \begin{lem}\label{quasi_homeo}
  The map $\chi\equiv \chi^{\nu_0,\nu_1}_{\alpha}$ satisfies the following properties.
    \begin{enumerate}[leftmargin=8mm]
         \item\label{real_anal_dep_item} Let $\mu_{\chi}:= \dfrac{\partial\chi/\partial \overline{z}}{\partial\chi/\partial{z}}$, wherever defined. Then for each fixed $\nu_0$ and $(r,\theta)$ with $r>0$ and $\theta\neq\nu_0$, the map 
         $$
         (\alpha,\nu_1) \mapsto \mu_{\chi^{\nu_0,\nu_1}_{\alpha}}(r,\theta)
         $$ 
         is real-analytic. Further, we have that $||\mu_{\chi}||_{\infty}<1$.
          \smallskip
          
       \item\label{chi_symm_item}  $\chi$ is a quasiconformal map that is symmetric with respect to $\mathbb{R}$; i.e., 
       \begin{equation}
           \chi(\overline{z})=\overline{\chi(z)}\quad \textrm{for}\ z \in \mathbb{C}^*.
           \label{chi_symmetry_eqn}
           \end{equation}
        
        \item if $\alpha=\dfrac{\log r}{\log R}-1$, then $\chi$ sends the disk $D_R$ (of radius $R$) onto the disk $D_r$.
        \smallskip

        \item $\chi(r,\nu_0)=(r^{1+\alpha},\nu_1)$ for all $r>0$.
        \smallskip
        
        \item\label{comm_diag} Let $\lambda_0 \in (0,1)$. Then we have the commutative diagram
\[        \begin{tikzcd}
D_R \arrow[r, "\chi"] \arrow[d, swap, "z \mapsto \lambda_0 z"] & D_r \arrow{d}{z \mapsto \lambda_1 z} \\
D_R \arrow[r, swap, "\chi"]& D_r
\end{tikzcd}\]
with $\lambda_1=\chi(\lambda_0)=\lambda_0^{1+\alpha}$. In particular, any $\lambda_1\in(0,1)$ can be attained by varying $\alpha>-1$.
    \end{enumerate}
\end{lem}
\begin{proof}
We begin with the explicit formulas for the partial derivatives of $\chi_\alpha^{\nu_0,\nu_1}$. For notational convenience, we suppress the parameters temporarily and use the notation $\chi\equiv\chi_\alpha^{\nu_0,\nu_1}$, $h\equiv h_{\nu_0}^{\nu_1}$, and $E(z) = e^{i h(\arg z)}$.

For $\theta< \nu_0 $, we have
$$
\dfrac{\partial \chi}{\partial x} = |z|^{\alpha-1} E(z) \left((1+\alpha)x - i s_1 y \right),
$$
$$
\dfrac{\partial \chi}{\partial y} = |z|^{\alpha-1} E(z) \left( (1+\alpha)y + i s_1 x \right).
$$ 
Hence for $\theta< \nu_0 $,
$$
\dfrac{\partial \chi}{\partial \bar{z}} = \dfrac{1}{2} \left( \dfrac{\partial \chi}{\partial x} + i \dfrac{\partial \chi}{\partial y} \right)= \dfrac{1}{2} |z|^{\alpha-1} E(z) (1+\alpha-s_1)(x + iy),\quad \textrm{and}
$$
$$
\dfrac{\partial \chi}{\partial z} = \dfrac{1}{2} \left( \dfrac{\partial \chi}{\partial x} - i \dfrac{\partial \chi}{\partial y} \right)= \dfrac{1}{2} |z|^{\alpha-1} E(z)  (1+\alpha+s_1)(x - iy).
$$
Thus, 
$$
\mu_{\chi_\alpha^{\nu_0,\nu_1}} (x+iy)= \dfrac{(1+\alpha-s_1) (x+iy)}{(1+\alpha+s_1) (x-iy)},\quad \textrm{when}\quad \theta< \nu_0.
$$ 
Since the change of variable between polar and cartesian coordinates is real-analytic, it is now clear that the  function $(\alpha,\nu_1) \mapsto \mu_{\chi^{\nu_0,\nu_1}_{\alpha}}(r,\theta)$ is a real-analytic function for each fixed $\nu_0$ and $(r,\theta)$ with $\theta<\nu_0$.
The above formula of $\mu_\chi$ also shows that
$$
|\mu_{\chi^{\nu_0,\nu_1}_{\alpha}}(r,\theta)| = \left| \dfrac{\partial \chi / \partial \bar{z}}{\partial \chi / \partial z} \right|= \dfrac{|1+\alpha-s_1| |x+iy|}{|1+\alpha+s_1| |x-iy|}= \dfrac{|1+\alpha-s_1|}{|1+\alpha+s_1|}<1,\quad \textrm{when}\quad \theta< \nu_0.
$$
Similar calculations for $\theta > \nu_0$ show that 
\begin{enumerate}[leftmargin=8mm]
    \item[(i)] $(\alpha,\nu_1) \mapsto \mu_{\chi^{\nu_0,\nu_1}_{\alpha}}(r,\theta)$ is real-analytic for each fixed $\nu_0$ and $(r,\theta)$ with $\theta>\nu_0$, and
    \item[(ii)] $|\mu_{\chi^{\nu_0,\nu_1}_{\alpha}}(r,\theta)| =\dfrac{\vert 1+\alpha-s_2\vert}{\vert 1+\alpha+s_2\vert}<1$, when $\theta>\nu_0$.
\end{enumerate}
This completes the proof of the first statement.

The second statement is a consequence of the bound on $||\mu_{\chi}||_{\infty}$ and the extension of $\chi$ from $\overline{\mathbb{H}}$ to $\C$ via Schwarz reflection. The third and fourth items follow by direct computations.
    
Let us denote multiplication by $\rho \in \mathbb{C}\setminus\ \{0\}$ as $ m_{\rho}(w)=\rho w$. The last statement follows from the relation 
$$
m_{\rho}\circ \chi = \chi\circ m_{\rho^{1/(1+\alpha)}}
$$
for $\rho\in\R\setminus\ \{0\}$.
\end{proof}

\subsection{Uniformization of tongues}\label{tongue_unif_subsec}

The main result of this section gives a dynamically natural parametrization for tongues in the complexified DSM family.
Let us fix a tongue $T$ of period $q$ and type $k/(2^q-1)$ for some  $k \in \{0,1,...,2^q-2\}$.

\begin{theorem}\label{parametrization_thm}
The map
\begin{equation*}
    \begin{split}
        \Xi: \Int{T}=T\cap\{(a,b): a\in\R/\Z,\ b\in (0,1)\} \to \mathbb{D}\setminus [0,1) \\
        \Xi(a,b) :=\lambda_{a,b}\dfrac{\kappa_{a,b}(c_1(a,b))}{\kappa_{a,b}(c_2(a,b))}\hspace{2cm}
    \end{split}
\end{equation*}
is a real-analytic diffeomorphism. In particular, $\Int{T}$ is simply connected.
\end{theorem}

The proof of Theorem~\ref{parametrization_thm} will be based on the following lemmas.
We fix a base-point $g_0:=g_{a_0,b_0}\in T$ for the rest of the section. Let $x_0$ be the distinguished attracting periodic point of $g_0$. We normalize the linearizing map $\kappa_0\equiv\kappa_{a_0,b_0}$ in a neighborhood of $x_0$ such that $\kappa_0(x_0)=0$ and $\kappa_0'(x_0)=C\dfrac{i}{x_0}$, for some $C>0$ (see Lemma~\ref{critical} (Item~\ref{koenig_sym}), and Remark~\ref{koenig_geometry_rem}). Assume further than the periodic point $x_0$ of $g_0$ has multiplier $\lambda_0$ and critical angle  $\nu_0$.

\begin{lem}\label{symm_linear}
  Using the notation of Lemma~\ref{critical} (Item~\ref{koenig_sym}) and Lemma~\ref{quasi_homeo}, we have that for $\kappa_0 : \mathcal{U} \to D_R $ and $\chi^{\nu_0,\nu_1}_{\alpha}: D_R \to D_r$,  the composition $\chi^{\nu_0,\nu_1}_{\alpha} \circ \kappa_0$ induces a Beltrami coefficient 
  $$
  \sigma_{\alpha,\nu_1}:= (\chi^{\nu_0,\nu_1}_{\alpha} \circ \kappa_0)^*(\mu_0)
  $$ 
  on $\mathcal{U}$ (where $\mu_0$ is the standard complex structure). For each fixed $(r,\theta)$ such that $\sigma_{\alpha,\nu_1}$ is defined, the Beltrami coefficient $\sigma_{\alpha,\nu_1}$ depends real-analytically on $\alpha, \nu_1$. Moreover, $\sigma_{\alpha,\nu_1}$ is invariant under $g_{0}^{\circ q}$ and $\sigma_{\alpha,\nu_1}\circ\eta(z)=\overline{\sigma_{\alpha,\nu_1}(z)}$.
\end{lem}
\begin{proof}
Recall that $\kappa_{0}$ conjugates $g_{0}^{\circ q}\vert_{\cU}$ to $m_{\lambda_0}\vert_{D_R}$, and $\chi^{\nu_0,\nu_1}_{\alpha}$ conjugates $m_{\lambda_0}\vert_{D_R}$ to $m_{\lambda_1}\vert_{D_r}$, where $r=R^{1+\alpha}$ and $\lambda_1=\lambda_0^{1+\alpha}$. Since $m_{\lambda_1}$ preserves the standard complex structure $\mu_0$, it now follows that $\sigma_{\alpha,\nu_1}=(\chi^{\nu_0,\nu_1}_{\alpha} \circ \kappa_0)^*(\mu_0)$ is invariant under $g_{0}^{\circ q}$.
The relation $\sigma_{\alpha,\nu_1}\circ\eta(z)=\overline{\sigma_{\alpha,\nu_1}(z)}$ is a consequence of real-symmetry of the standard complex structure and Relations~\eqref{kappa_symmetry_eqn} and~\eqref{chi_symmetry_eqn}.

The real-analytic dependence of the Beltrami coefficient $\sigma_{\alpha,\nu_1}$ on the parameters follows from lemma \ref{quasi_homeo} (Item~\ref{real_anal_dep_item}).
\end{proof}

The next lemma is key to the proof of surjectivity of $\Xi$.

\begin{lem}\label{qc_def_lemma}
   For any $\lambda_1\in(0,1)$ and $\nu_1 \in (0,\pi)$, there exists a quasiconformal homeomorphism $\Phi_{\lambda_1,\nu_1}$  of the Riemann sphere $\widehat{\C}$ satisfying the following properties.
    \begin{itemize}[leftmargin=8mm]
        \item $\Phi_{\lambda_1,\nu_1}$ fixes $0$ and $\infty$.
        \item $\Phi_{\lambda_1,\nu_1}\circ\eta=\eta\circ\Phi_{\lambda_1,\nu_1}$.
        \item $\Phi_{\lambda_1,\nu_1}\circ g_0 \circ \Phi_{\lambda_1,\nu_1}^{-1}\in T$ has an attracting $q-$cycle on $\mathbb{S}^1$ with multiplier $\lambda_1$ and critical angle $\nu_1$.
    \end{itemize}
\end{lem}
\begin{proof}
\captionsetup{width=0.98\linewidth}
    \begin{figure}
    \centering
    \includegraphics[width=0.96\linewidth]{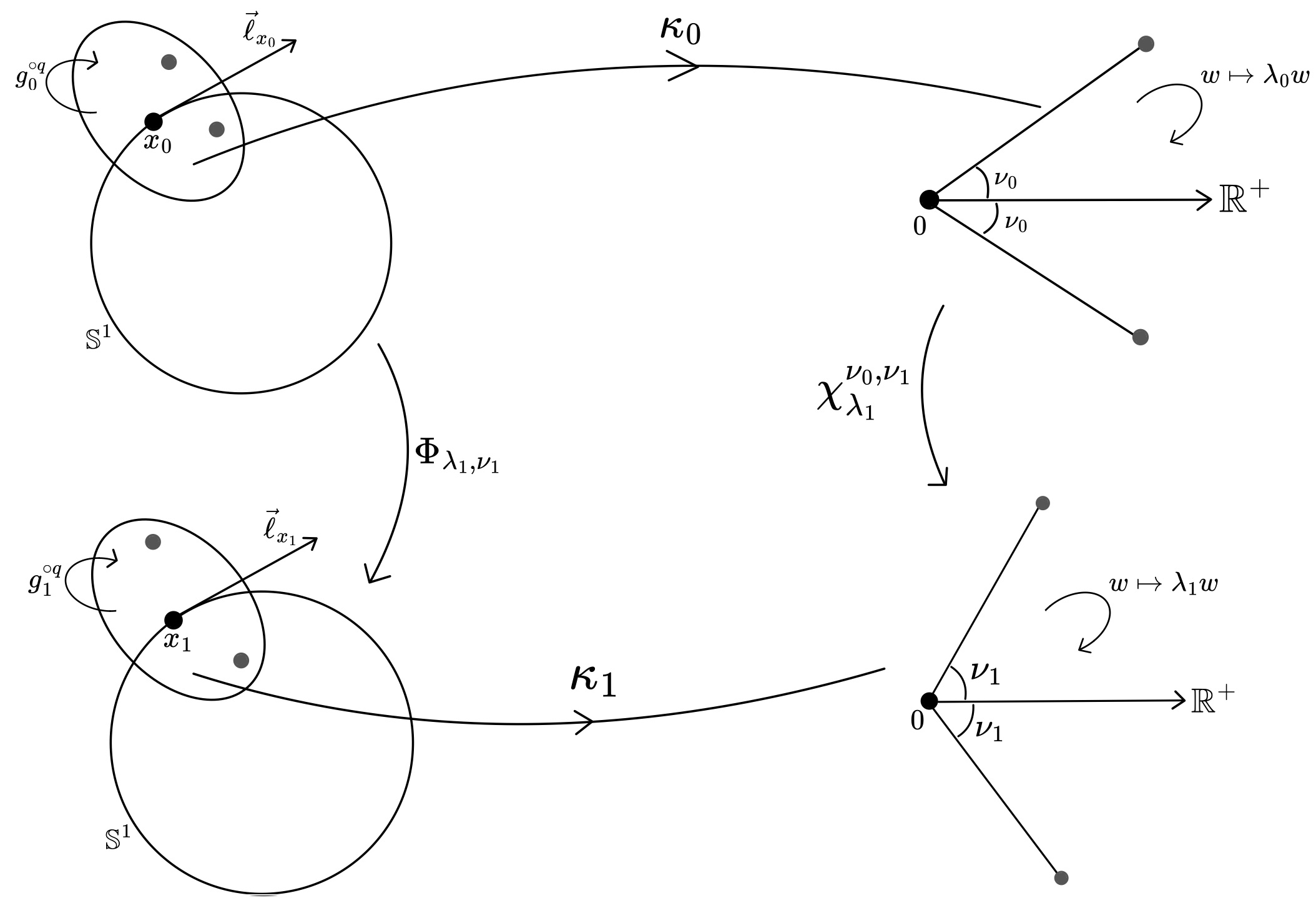}
    \caption{Depicted is a schematic diagram of the quasiconformal deformation of Lemma~\ref{qc_def_lemma}. The red points on the left are the critical points of $g_0, g_1$, and those on the right are the images of the critical points under the corresponding linearizing maps.}
    \label{qc_def_fig}
\end{figure}
Let $\chi_{\lambda_1}^{\nu_0,\nu_1}\equiv\chi_{\alpha}^{\nu_0,\nu_1}$ be the quasiconformal homeomorphism of Lemma~\ref{quasi_homeo} where $\lambda_0^{1+\alpha}=\lambda_1$. By our choice of $\alpha$, the real numbers $\alpha$ and $\lambda_1$ are in one-to-one correspondence via the relation $\alpha=\dfrac{\ln \lambda_1}{\ln\lambda_0}-1$. Let $\sigma_{\lambda_1,\nu_1}\equiv \sigma_{\alpha,\nu_1}$ be the Beltrami coefficient on $\cU$ as defined in Lemma~\ref{symm_linear}, where we recall that $\cU$ is an $\eta$-symmetric domain of linearization containing the distinguished attracting periodic point $x_0$ of~$g_0$.

 Now pull back $\sigma_{\lambda_1,\nu_1}$ to the entire basin of attraction of the unique attracting cycle of $g_0$ using the dynamics and extend it to $\widehat{\C}$  as $\mu=0$. This produces a $g_0$-invariant Beltrami coefficient $\sigma_{\lambda_1,\nu_1}$ on $\widehat{\C}$ that depends real-analytically on $\alpha,\nu_1$ for each fixed $(r,\theta)$ (whenever the Beltrami coefficient is defined; see Lemma~\ref{symm_linear}). Further, as $\sigma_{\lambda_1,\nu_1}\vert_\cU$ is $\eta$-symmetric (by Lemma~\ref{symm_linear}) and $g_0$ commutes with $\eta$, we conclude that $\sigma_{\lambda_1,\nu_1}\circ\eta(z)=\overline{\sigma_{\lambda_1,\nu_1}(z)}$ a.e.~on $\widehat{\C}$. By the Measurable Riemann Mapping Theorem, there exist quasiconformal homeomorphisms  $\widecheck{\Phi}_{\lambda_1,\nu_1}$ fixing $0$ and $\infty$ such that 
 \begin{enumerate}[leftmargin=8mm]
     \item $\overline{\partial}\widecheck{\Phi}_{\lambda_1,\nu_1}/\partial\widecheck{\Phi}_{\lambda_1,\nu_1} = \sigma_{\lambda_1,\nu_1}$ a.e., and
     
     \item $\widecheck{\Phi}_{\lambda_1,\nu_1}$ commutes with $\eta$ (this follows from the $\eta$-symmetry of $\sigma_{\lambda_1,\nu_1}$).
 \end{enumerate}
Note that the relation $\widecheck{\Phi}_{\lambda_1,\nu_1}\circ\eta=\eta\circ\widecheck{\Phi}_{\lambda_1,\nu_1}$ implies that $\widecheck{\Phi}_{\lambda_1,\nu_1}$ preserves the unit circle. As $\widecheck{\Phi}_{\lambda_1,\nu_1}$ also fixes $0$ and $\infty$ by normalization, it follows that $\widecheck{\Phi}_{\lambda_1,\nu_1}$ carries the unit disk (respectively, the exterior of the unit disk) to itself.
Conjugating $g_{0}$ via $\widecheck{\Phi}_{\lambda_1,\nu_1}$, we get a holomorphic map, where the holomorphicity follows from the $g_0$-invariance of $\sigma_{\lambda_1,\nu_1}$ and the Weyl's lemma. According to Lemma~\ref{deformation_in_sm}, there exists a unique rotation $R_\tau(w)=e^{i\tau}w$ ($\tau\in\R$) such that if we set $\Phi_{\lambda_1,\nu_1}:=R_\tau\circ\widecheck{\Phi}_{\lambda_1,\nu_1}$, then the holomorphic map 
$$
g_1:=\Phi_{\lambda_1,\nu_1}\circ g_0\circ \Phi_{\lambda_1,\nu_1}^{-1}
$$ 
lies in the complexified DSM family. We note that with the above normalization, the quasiconformal map $\Phi_{\lambda_1,\nu_1}$ is uniquely defined, and is also $\eta$-symmetric (as $\eta$ commutes with $R_\tau$).

By construction, the map $g_1$, which lies in the complexified DSM family, has an attracting $q$-cycle on $\mathbb{S}^1$ and two distinct $\eta$-symmetric critical points. We denote the distinguished attracting periodic point $\Phi_{\lambda_1,\nu_1}(x_0)$ of $g_1$ by $x_1$. We claim that a linearizing map for $g_1^{\circ q}$ in a neighborhood of $x_1$ is given by 
$$
\kappa_1:=\chi_{\lambda_1}^{\nu_0,\nu_1}\circ\kappa_0\circ\Phi_{\lambda_1,\nu_1}^{-1}
$$
To see this, note that the quasiconformal map $\chi_{\lambda_1}^{\nu_0,\nu_1}\circ\kappa_0$ pulls back the trivial Beltrami coefficient $\mu_0$ in a neighborhood of the origin to the Beltrami coefficient $\sigma_{\lambda_1,\nu_1}$ in a neighborhood of $x_0$, and the quasiconformal map $\Phi_{\lambda_1,\nu_1}^{-1}$ pulls back $\sigma_{\lambda_1,\nu_1}$ in a neighborhood of $x_0$ to the trivial Beltrami coefficient in a neighborhood of $x_1$. Therefore, the quasiconformal map $\chi_{\lambda_1}^{\nu_0,\nu_1}\circ\kappa_0\circ\Phi_{\lambda_1,\nu_1}^{-1}$ pulls back the trivial Beltrami coefficient in a neighborhood of the origin to the trivial Beltrami coefficient in a neighborhood of $x_1$. It now follows by the Weyl's lemma that $\chi_{\lambda_1}^{\nu_0,\nu_1}\circ\kappa_0\circ\Phi_{\lambda_1,\nu_1}^{-1}$ is a conformal map from a neighborhood of $x_1$ to a neighborhood of the origin. It also follows from the construction that $\kappa_1=\chi_{\lambda_1}^{\nu_0,\nu_1}\circ\kappa_0\circ\Phi_{\lambda_1,\nu_1}^{-1}$ is a conformal conjugacy between the restriction of $g_1^{\circ q}$ in a neighborhood of $x_1$ to the restriction of the linear map $w\mapsto \lambda_1 w$ in a neighborhood of the origin (see Lemma~\ref{quasi_homeo}, Item~\ref{comm_diag}) and Figure~\ref{qc_def_fig}). Thus, $\kappa_1$ is a linearizing map, as claimed.

The above discussion implies that the multiplier of the attracting $q$-cycle of $g_1$ is equal to $\lambda_1$. Note also that the quasiconformal map $\Phi_{\lambda_1,\nu_1}$ commutes with $\eta$ (reflection in $\mathbb{S}^1)$, the normalized linearizing map $\kappa_0$ conjugates $\eta$ to the complex conjugation map, and the quasiconformal map $\chi_{\lambda_1}^{\nu_0,\nu_1}$ commutes with the complex conjugation map (see Lemma~\ref{quasi_homeo}, Item~\ref{chi_symm_item}). Hence, $\kappa_1$ sends $x_1$ to $0$, and conjugates $\eta$ to the complex conjugation map. Thus, by Lemma~\ref{critical} (Item~\ref{koenig_sym}), we get that $\kappa_1'(x_1)=C\dfrac{i}{x_1}$ for some $C \in \mathbb{R}\setminus\{0\}$. The fact that $\kappa_1$ sends points near $x_1$ lying in the exterior of the unit disk to points in the upper half-plane also shows that $\kappa_1$ sends the tangent direction $\vec{\ell}_{x_1}$ to the positive real direction at the origin, and hence we have $C>0$ (see Remark~\ref{koenig_geometry_rem}). By Definition~\ref{crit_angle_def} (cf. Remark~\ref{crit_anglr_rem}), the critical angle of $g_1$ is given by $\arg{\kappa_1(c_1(g_1))}$, where $c_1(g_1)$ is the critical  point of $g_1$ that lies in the exterior of the unit disk. Since the quasiconformal conjugacy $\Phi_{\lambda_1,\nu_1}$ between $g_0$ and $g_1$ respects the critical points of the maps, and  $\chi_{\lambda_1}^{\nu_0,\nu_1}$ takes points with argument $\nu_0$ to points with argument $\nu_1$, it is now evident that $\arg{\kappa_1(c_1(g_1))}=\nu_1$ (see Figure~\ref{qc_def_fig}). We conclude that the critical angle of $g_1$ is equal to $\nu_1$.

Finally, the arguments of \cite[Proposition~3.7]{D} can be applied verbatim to the current setting to show that the type of the attracting $q$-cycle of $g_1$ is equal to that of $g_0$. Hence, $g_1\in T$.
\end{proof}

\begin{remark}\label{trivial_qc_def_rem}
Let $(\lambda_1,\nu_1)=(\lambda_0,\nu_0)$. Then, the map $\chi_{\lambda_1}^{\nu_0,\nu_1}$ is readily seen to be the identity map, and hence $\sigma_{\lambda,\nu_1}$ turns out to be the trivial Beltrami coefficient. Thus, in this case, the map $\Phi_{\lambda_1,\nu_1}$ is a M{\"o}bius map fixing $0,\infty$, and preserving the unit circle; i.e., $\Phi_{\lambda_1,\nu_1}$ is a rotation around the origin. Further, as $g_0$ already lies in the the complexified DSM family, the requirement that $\Phi_{\lambda_1,\nu_1}\circ g_0\circ \Phi_{\lambda_1,\nu_1}^{-1}$ also belongs to the complexified DSM family forces $\Phi_{\lambda_1,\nu_1}$ to be the identity map.
\end{remark}

\begin{lem}\label{onto_analytic}
   $\Xi$ is a surjective, real-analytic map.
\end{lem}
\begin{proof}
It is easily seen that for $(a,b)\in T$, if $g_{a,b}$ has multiplier $\lambda_{a,b}$ and critical angle $\nu_{a,b}$ then $\Xi(a,b)=\lambda_{a,b}\dfrac{\kappa_{a,b}(c_1(a,b))}{\kappa_{a,b}(c_2(a,b))}=\lambda_{a,b} e^{2i\nu_{a,b}}$, where $c_1(a,b)$ and $c_2(a,b)$ are the critical points of $g_{a,b}$ outside and inside $\mathbb{S}^1$, respectively.

Since $g_{a,b}$ depends real-analytically on $a,b$, it follows that the linearizing map $\kappa_{a,b}$ also depends real-analytically on $a,b$ (cf. \cite[Remark~8.3]{M}). The explicit formulas for $c_1(a,b), c_2(a,b)$ given in Lemma~\ref{critical} show that the critical points of $g_{a,b}$ are real-analytic functions of $a,b$. Finally, real-analyticity of $(a,b) \mapsto \lambda_{a,b}$ follows by a simple application of the implicit function theorem.
Real-analyticity of $\Xi$ follows from the previous observations. 

Let us now pick $\lambda_1 e^{2i\nu_1}\in \D\setminus[0,1)$, where $\nu_1\in(0,\pi)$. By Lemma~\ref{qc_def_lemma}, there exists $g_{\widetilde{a},\widetilde{b}}\in T$ having an attracting $q$-cycle on $\mathbb{S}^1$ with multiplier $\lambda_1$ and critical angle $\nu_1$. Hence, $\Xi(\widetilde{a},\widetilde{b})=\lambda_1 e^{2i\nu_1}$. This proves surjectivity of~$\Xi$.
\end{proof}

The following lemma will be useful in proving injectivity of~$\Xi$.
The crux of the proof of the lemma is to construct local inverse branches to the map $\Xi$ using the quasiconformal deformation designed in Lemma~\ref{qc_def_lemma}. In fact, this idea goes back to the proof of multiplier uniformizations of hyperbolic components of the Mandelbrot set (cf. \cite[Theorem~4.1, p. 166]{BF}).

\begin{lem}\label{Xi_covering}
    $\Xi$ is a covering map.
\end{lem}
\begin{proof}
    We have already established surjectivity of $\Xi$. 
    Now let $\Xi(a,b)=\lambda e^{2i\nu}$. We will construct a continuous local inverse of $\Xi$ in a neighborhood of $\lambda e^{2i\nu}$.

    Fix $\epsilon>0$ small enough. For real $\lambda_1, \nu_1$ satisfying $\vert\lambda_1-\lambda\vert<\epsilon$ and $\vert\nu_1-\nu\vert<\epsilon$, let $\chi_{\lambda_1}^{\nu,\nu_1}\equiv\chi_{\alpha}^{\nu,\nu_1}$ be the quasiconformal map of Lemma~\ref{quasi_homeo} (where $\lambda^{1+\alpha}=\lambda_1$), and $\sigma_{\lambda_1,\nu_1}=\sigma_{\alpha,\nu_1}$ be the $g_{a,b}-$invariant Beltrami coefficient produced in Lemma~\ref{qc_def_lemma}. Further, let $\Phi_{\lambda_1,\nu_1}$ be the normalized quasiconformal homeomorphism of Lemma~\ref{qc_def_lemma} that solves the Beltrami equation with coefficient $\sigma_{\lambda_1,\nu_1}$.
    
  We now  consider the map
    \[
   \mathfrak{I}(\lambda_1e^{2i\nu_1})= (a_1,b_1)\in T,
    \]  
where $g_{a_1,b_1}=\Phi_{\lambda_1,\nu_1}\circ g_{a,b}\circ \Phi_{\lambda_1,\nu_1}^{-1}$. By Lemma~\ref{symm_linear}, the Beltrami coefficient $\sigma_{\lambda_1,\nu_1}$ varies real-analytically with $(\lambda_1,\nu_1)$, and hence the parametric Measurable Riemann Mapping Theorem implies that the map $\mathfrak{I}$ is also real-analytic (see \cite[\S 1.4.1]{BF}, cf. \cite[Appendix~2, p. 124, Lemma~6]{Ahl06}). Note also that $\mathfrak{I}(\lambda e^{2i\nu})=(a,b)$ (see Remark~\ref{trivial_qc_def_rem}).

By the proof of Lemma~\ref{qc_def_lemma}, the map $g_{a_1,b_1}$ has an attracting $q$-cycle with multiplier $\lambda_1$ and critical angle $\nu_1$. Since a map in the complexified DSM family has a unique attracting cycle, it follows that $\Xi\circ\mathfrak{I}$ is the identity map on the neighborhood $\{\lambda_1 e^{2i\nu_1}: \vert\lambda_1-\lambda\vert<\epsilon,\  \vert\nu_1-\nu\vert<\epsilon\}$ of $\lambda e^{2i\nu}$. Thus, the map $\Xi$ is a left inverse of $\mathfrak{I}$, proving injectivity of $\mathfrak{I}$. By the Invariance of Domain Theorem (cf. \cite[Theorem~2B.3]{Hat02}), $\mathfrak{I}$ is a homeomorphism from a neighborhood of $\lambda e^{2i\nu}$ onto a neighborhood of $(a,b)$. We conclude that $\mathfrak{I}$ is the desired local inverse of $\Xi$.
\end{proof}

\begin{proof}[Proof of Theorem~\ref{parametrization_thm}]
    Combining Lemma~\ref{onto_analytic} and Lemma~\ref{Xi_covering}, we have that $\Xi$ is a covering map of the simply connected domain $\mathbb{D}\setminus[0,1)$. So $\Xi$ is a homeomorphism. It was also demonstrated in these lemmas that $\Xi$ is real-analytic and so are its local inverses. Thus, $\Xi$ is a real-analytic diffeomorphism. 
\end{proof}

\begin{remark}\label{xi_extends_rem}
    Observe that in the statement of Theorem~\ref{parametrization_thm} we only uniformize $\Int{T}$; i.e., the part of the tongue $T$ obtained by removing the ceiling $\{b=1\}$. However, it is easy to extend this uniformization to the ceiling $T \cap \{b=1\}$. Such an extension of $\Xi$ will be a two-to-one map from $T \cap \{b=1\}$ onto the slit $[0,1)$, branched at the unique superattracting parameter in $T\cap \{b=1\}$. 
\end{remark}

\begin{remark}\label{internal_rays_land_rem}
For $\nu\in(0,\pi)$, the \emph{internal ray} of $T$ at angle $\nu$ is defined as the preimage of the radial line at angle $\nu$ under the uniformization $\Xi$. We note that these are precisely the curves constructed in \cite{D}. The arguments of \cite[Lemmas~6.3,~6.4]{IM} (building on a result relating Koenigs linearizations and parabolic Fatou coordinates, see \cite[Theorem~1.2]{Kaw}) can be adapted for the current setting to show that the internal rays of $T$ at angles $\nu\in(0,\pi)$ land at the tip of $T$ (i.e., at the unique double parabolic parameter on $\partial T$, cf. \cite{BBCE}) as the multiplier tends to $1$.
\end{remark}

The quasiconformal deformation techniques used in the previous lemmas immediately give a direct path connecting any pair of geometrically attracting parameters in $\Int{T}$ (recall that the paths constructed in \cite{D} connect each geometrically attracting parameter in $T$ to the unique superattracting parameter in $T\cap\{b=1\}$ such that all parameters on such a path have the same critical angle). We record this fact for completeness.

\begin{lem}\label{direct_path_lem}
Consider two geometrically attracting maps $g_i$ in $T$ with multipliers $\lambda_i$ and critical angles $\nu_i$, $i\in\{0,1\}$. Then, there exists a path connecting $g_0, g_1$ in~$\Int{T}$.
\end{lem}
\begin{proof}[Sketch of proof]
Consider the paths 
$$
\lambda(t)=(1-t)\lambda_0+t\lambda_1\quad \mathrm{and}\quad \nu(t)=(1-t)\nu_0+t \nu_1,\ t\in [0,1].
$$
For $t\in[0,1]$, let $\chi_{\lambda(t)}^{\nu_0,\nu(t)}:=\chi_{\alpha(t)}^{\nu_0,\nu(t)}$ be as in Lemma~\ref{quasi_homeo}, where $\alpha(t)=\dfrac{\ln \lambda(t)}{\ln \lambda_0}-1$. Then, the Beltrami coefficients $\sigma_{\lambda(t),\nu(t)}\equiv\sigma_{\alpha(t),\nu(t)}:= (\chi^{\nu_0,\nu(t)}_{\lambda(t)} \circ \kappa_0)^*(\mu_0)$ depend real-analytically (in particular, continuously) on $t$ (see Lemma~\ref{symm_linear}). The quasiconformal deformation argument of Lemma~\ref{qc_def_lemma}, combined with the parametric Measurable Riemann Mapping Theorem, yields a path 
$$
\Phi_{\lambda(t),\nu(t)}\circ g_0\circ\Phi_{\lambda(t),\nu(t)}^{-1},\ t\in[0,1],
$$ 
in $T$ connecting $g_0$ to $g_1$ (where $\Phi_{\lambda(t),\nu(t)}$ is the normalized quasiconformal map solving the Beltrami equation with coefficient $\sigma_{\lambda(t),\nu(t)}$, cf. Remark~\ref{trivial_qc_def_rem}).     
\end{proof}

\begin{cor}\label{qc_conjugate_cor}
Let $(a_1,b_1), (a_2,b_2)\in \Int{T}$. Then, the maps $g_{a_1,b_1}$ and $g_{a_2,b_2}$ are quasiconformally conjugate.
\end{cor}
\begin{proof}
    This follows from Lemma~\ref{qc_def_lemma} and bijectivity of $\Xi$.
\end{proof}

\section{Parameter dependence of maximal chaotic set}

 The goal of this section is to study the parameter dependence of a maximal subset of $\mathbb{S}^1$ where $g_{a,b}$, $(a,b)\in\cP$, is Devaney chaotic. 
We begin by recalling the notion of Devaney chaos.
    
\begin{defn}\label{devaney_chaos_def}
    Let $f:X\to X$ be a continuous map. We say that $f$ is Devaney chaotic if, 
\begin{enumerate}[leftmargin=8mm]
		\item  $f$ is topologically transitive; i.e., for any two non-empty open sets $U$ and $V$ of $X$ there exists a natural number $n\in \mathbb{N}$ such that $f^{\circ n}(U)\cap V\ne \emptyset$; and
        \vspace{.1cm}
         \item the set of all periodic points of $f$ is dense in $X$.
         \end{enumerate}
\end{defn}

\subsection{Maximal set of Devaney chaos}\label{maximal_chaos_subsec}
 Recall that the hyperbolic locus $\cH$ of complexified Double Standard Maps consists of parameters in $\cP$ for which the corresponding complexified Double Standard Map has an attracting~cycle.
 
 \subsubsection{Parameters outside hyperbolic closure}
  We start with the simple observation that when $(a,b)\notin\overline{\cH}$, the map $g_{a,b}$ is chaotic on the entire unit circle.
\begin{prop}\label{outside_tongue}
   If $(a,b) \notin \overline{\mathcal{H}}$ then $g_{a,b}:\mathbb{S}^1\to\mathbb{S}^1$ is Devaney chaotic.
\end{prop}
\begin{proof}
   Observe that $g_{a,b}|_{\mathbb{S}^1}$ has no critical points if $b<1$. Further, if $(a,b) \notin  \overline{\mathcal{H}}$, then $g_{a,b}|_{\mathbb{S}^1}$ has no neutral cycle on $\mathbb{S}^1$. Hence, by a theorem Ma{\~n}{\'e} \cite[Chapter~III, Theorem~5.1]{MS}, there exist $C>0,\lambda>1$ such that for any  $n$,  $|(g_{a,b}^{\circ n})'(z)|>C \lambda^{n} $ for all $z \in \mathbb{S}^1$. Choose $m$ large so that $C\lambda^m >1$ implying that  $g_{a,b}^{\circ m}:\mathbb{S}^1\to\mathbb{S}^1$ is distance expanding.
   It follows that $g_{a,b}^{\circ m}$ is topologically transitive (cf. \cite[Proposition~4.7]{CK}). In particular, $g_{a,b}^{\circ m}$ has a dense orbit on $\mathbb{S}^1$. Since $g_{a,b}^{\circ m}$ is not injective on $\mathbb{S}^1$, we can now appeal to \cite[Theorem~7.1]{S} to conclude that $g_{a,b}^{\circ m}$ has a dense set of periodic points.
   Clearly, the properties of topological transitivity and density of periodic points are inherited by $g_{a,b}\vert_{\mathbb{S}^1}$.
\end{proof}

\subsubsection{Parameters in tongues}
We now turn our attention to parameters inside tongues. For the remainder of this subsection, we fix a parameter $(a,b)$ in a tongue of the complexified DSM family; i.e., $g_{a,b}$ has an attracting cycle on $\mathbb{S}^1$. We define
$$
\mathcal{B}_{\infty}:= \left(\bigcup_{j=0}^{\infty}g_{a,b}^{-j}(B_0)\right)\cap\mathbb{S}^1,\qquad C_{a,b} = \mathbb{S}^1\ \setminus\ \mathcal{B}_\infty,
$$
where $B_0$ is the immediate basin of attraction of the unique attracting cycle of $g_{a,b}$.  The set $\mathcal{B}_\infty$ is the total basin of attraction of the unique attracting cycle in $\mathbb{S}^1$. We record the following facts, which follow readily from the definitions.

\begin{lem}\label{inv_part_lem}
The sets $\mathcal{B}_\infty$ and $C_{a,b}$ satisfy the following properties.
 \noindent\begin{enumerate}[leftmargin=8mm]
        \item $\mathbb{S}^1=\mathcal{B}_\infty\sqcup C_{a,b}$, where $\mathcal{B}_\infty$ is open and $C_{a,b}$ is closed in $\mathbb{S}^1$.
        \item $C_{a,b}$ and $\mathcal{B}_\infty$ are completely invariant under $g_{a,b}:\mathbb{S}^1\to\mathbb{S}^1$.
    \end{enumerate}    
\end{lem}

We refer the reader to \cite[Chapter~III]{MS} for the notion of a hyperbolic set.
\begin{lem}\label{hyperbolic}
    The following statements are true.
    \begin{itemize}[leftmargin=8mm]
    \item $C_{a,b}$ is a hyperbolic set for $g_{a,b}$. The one-dimensional Lebesgue measure of $C_{a,b}$ is zero, and hence $C_{a,b}$ is nowhere dense in $\mathbb{S}^1$.
    \smallskip
    
    \item If $g_{a,b}$ is Devaney chaotic on some closed, forward invariant subset $X$ of $\mathbb{S}^1$, then $X$ is contained in the complement of entire basin of attraction; i.e.,  $X \subset C_{a,b}$. In particular, $X$ has one-dimensional Lebesgue measure $0$.
    \end{itemize}
\end{lem}
\begin{proof}
 By construction, $C_{a,b}$ is a compact, forward invariant set. Since $g_{a,b}$ has an attracting cycle on $\mathbb{S}^1$, both the critical points of $g_{a,b}$ lie in the immediate basin $B_0$, and hence $g_{a,b}$ has no other non-repelling cycle. According to \cite[Chapter~III, Corollary~5.1]{MS}, the set $C_{a,b}$ is hyperbolic for $g_{a,b}$. By \cite[Chapter~III, Theorem~2.6]{MS}, $C_{a,b}$ is either of full measure or zero measure. Since $C_{a,b}$ is disjoint from the open set $B_0\cap\mathbb{S}^1$, it follows that $C_{a,b}$ has measure zero.
\smallskip

 The next point is immediate as a set of Devaney chaos cannot intersect the basin of attraction of $g_{a,b}$.    
\end{proof}

The above result motivates the following definition.
\begin{defn}\label{max_chaos_def}
A closed, forward invariant set $X\subset\mathbb{S}^1$ is said to be a \emph{maximal chaotic set} for $g_{a,b}:\mathbb{S}^1\to\mathbb{S}^1$ if $g_{a,b}\vert_X$ is Devaney chaotic, and there is no $Y\subset\mathbb{S}^1$ such that the following two conditions are satisfied.
\begin{itemize}[leftmargin=8mm]
    \item $X\subsetneq Y$.
    \item $Y$ is forward invariant under $g_{a,b}$ and $g_{a,b}\vert_Y$ is Devaney chaotic.
\end{itemize}
\end{defn}

We will show that $C_{a,b}$ is indeed a maximal chaotic set for $g_{a,b}$. Thanks to Lemma~\ref{hyperbolic}, this will follow if we prove that $g_{a,b}$ is Devaney chaotic on $C_{a,b}$. The next result can be regarded as a precursor to Devaney chaos of $g_{a,b}\vert_{C_{a,b}}$. 

    \begin{lem}\label{julia_intesect}
      Let $J_{a,b}$ be the Julia set of $g_{a,b}$. Then $C_{a,b}= J_{a,b}\cap\mathbb{S}^1$.
    \end{lem}
    \begin{proof}
        Since the basin of attraction is contained in the Fatou set of $g_{a,b}$, we have that $J_{a,b}\cap\mathbb{S}^1\subset C_{a,b}$.

        The opposite containment is a consequence of the fact that $C_{a,b}$ is a hyperbolic set for $g_{a,b}$. Indeed, the unbounded growth of derivatives of $g_{a,b}^{\circ n}$, $n\in\N$, implies that the iterates of $g_{a,b}$ cannot form a normal family in a neighborhood of a point of $C_{a,b}$. Alternatively, every neighborhood of a point of $C_{a,b}$ intersects the attracting basin of $g_{a,b}$ (as $C_{a,b}$ is nowhere dense), and hence $C_{a,b}$ lies in the boundary of the attracting basin.
    \end{proof}

Recall that according to \cite[Lemma~3.2]{MR}, there is a monotone non-decreasing map $\varphi_{a,b}:\mathbb{S}^1\to\mathbb{S}^1$ that semi-conjugates $g_{a,b}$ to the doubling map $D$. Since $(a,b)$ lies in a tongue, $\varphi_{a,b}$ collapses $B_0\cap\mathbb{S}^1$ (where $B_0$ is the immediate basin of attraction of the unique attracting cycle of $g_{a,b}$) to a cycle $\mathfrak{C}$ of the doubling map. By the semi-conjugacy relation, $\varphi_{a,b}$ collapses $\mathcal{B}_\infty$ to the grand orbit $\widehat{\mathfrak{C}}$ of the cycle $\mathfrak{C}$ (note that the grand orbit $\widehat{\mathfrak{C}}$ consists of $\mathfrak{C}$ and its iterated preimages under the doubling map $D$). By continuity, the closure of each component of $\mathcal{B}_\infty$ is mapped by $\varphi_{a,b}$ to a point in this grand orbit. In fact, the converse is also true.

\begin{lem}\label{phi_preimage_lem}
Let $p\in\widehat{\mathfrak{C}}$. Then, $\varphi_{a,b}^{-1}(p)$ is the closure of a component of $\mathcal{B}_\infty$.     
\end{lem}
\begin{proof}
Since $p\in\widehat{\mathfrak{C}}$, there exists $n\geq 0$ such that $q:=D^{\circ n}(p)\in\mathfrak{C}$. Set $I:=\varphi_{a,b}^{-1}(p)$, which is a (possibly degenerate) interval by the monotonicity of $\varphi_{a,b}$. Since $\varphi_{a,b}$ is a semi-conjugacy from $g_{a,b}$ to $D$, we have that $g_{a,b}^{\circ n}(I)\subset\varphi_{a,b}^{-1}(q)$.

By \cite[Lemma~4.1]{MR}, the fiber $\varphi_{a,b}^{-1}(q)$ is the closure of a component $B_{0}^{1}$ of $B_0\cap\mathbb{S}^1$, and the endpoints of $\overline{B_0^1}$ are repelling periodic points of $g_{a,b}$ lying in $J_{a,b}$. Since $g_{a,b}^{\circ n}(I)\subset\overline{B_0^1}$, it follows by the discussion before the statement of this lemma that $I$ contains $\overline{B_n}$, where $B_n$ is a component of $g_{a,b}^{-n}(B_0^1)$ and hence also a component of $\mathcal{B}_\infty$. 
We claim that $I=\overline{B_n}$. By way of contradiction, assume that $I\supsetneq\overline{B_n}$. The iterate $g_{a,b}^{\circ n}$ maps $\overline{B_n}$ onto $\overline{B_0^1}$ (since $g_{a,b}$ maps Fatou components onto Fatou components).  In particular, this implies that $g_{a,b}^{\circ n}(I)=\overline{B_0^1}$. The assumption $I\supsetneq\overline{B_n}$ now implies that there is a folding under $g_{a,b}^{\circ n}$ at one of the endpoints of $\overline{B_n}$; i.e, $g_{a,b}^{\circ n}$ has a critical point on $\partial B_n\in J_{a,b}$. But this forces $J_{a,b}$ to contain a preimage of a critical point of $g_{a,b}$, which is impossible since both critical points of $g_{a,b}$ lie in the distinguished Fatou component. This contradiction proves that $I=\overline{B_n}$.
\end{proof}

\begin{prop}\label{transitive_prop}
     $g_{a,b}:C_{a,b}\to C_{a,b}$ is topologically transitive.  
     \end{prop}
\begin{proof}
    Let $I_1\cap C_{a,b}$, $I_2\cap C_{a,b}$ be two non-empty open sets in $C_{a,b}$ where $I_1,I_2$ are open intervals in $\mathbb{S}^1$. 
    We recall that $\mathcal{B}_\infty$ and $C_{a,b}$ yield a dynamically invariant partition of $\mathbb{S}^1$ into an open set and a closed set (respectively), and that the normalized one-dimensional Lebesgue measure of $\mathcal{B}_\infty$ (respectively, of $C_{a,b}$) is $1$ (respective, $0$). 
Since $C_{a,b}$ is nowhere dense, both $I_1,I_2$ must intersect $\mathcal{B}_{\infty}$. Let $J_1,J_2$ be components of $\mathcal{B}_\infty$ that intersect $I_1, I_2$ non-trivially, respectively. Let $J_i=(l_i,r_i)$, where $l_i<r_i$, $i\in\{1,2\}$. Since $J_1, J_2$ are connected components of $\mathcal{B}_{\infty}$, Lemma ~\ref{inv_part_lem} implies that $l_1,l_2,r_1,r_2 \in C_{a,b}$. Since $I_i$ intersects $C_{a,b}$ non-trivially, it must contain at least one endpoint of $J_i$, $i\in\{1,2\}$. We assume that $l_i\in I_i$, $i\in\{1,2\}$, the other cases being similar (see Figure~\ref{intervals_fig}).
 \begin{figure}[h!] 
  \captionsetup{width=0.98\linewidth}
     \centering
    \includegraphics[width=.96\linewidth]{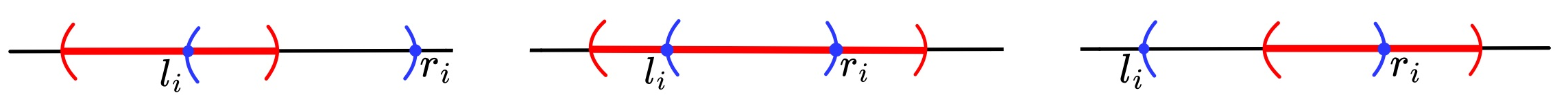}
    \caption{Displayed are the possible configurations of the intervals $I_i$ and $J_i$. The red intervals are $I_i$, while $J_i=(l_i,r_i)$.}
    \label{intervals_fig}
\end{figure}

 By  the discussion preceding  Lemma~\ref{phi_preimage_lem}, the semi-conjugacy $\varphi_{a,b}$ between $g_{a,b}|_{\mathbb{S}^1}$ and $D$ is constant on any connected component  of $\mathcal{B}_\infty$. Let $\varphi \equiv p_1$ on $[l_1,r_1]$ and $\varphi \equiv p_2$ on $[l_2,r_2]$. 
Recall that the iterated preimages of a point under the doubling map $D$ are dense in $\mathbb{S}^1$. Hence, there exist a  sequence of natural numbers $m_k \to \infty$ and an increasing sequence $\{w_k\} \subset \mathbb{S}^1$ such that $w_k \uparrow p_1$ with $D^{\circ m_k}(w_k)=p_2$. 

By Lemma~\ref{phi_preimage_lem}, the preimage $\varphi_{a,b}^{-1}(w_k)$ is the closure of a component $B_k$ of $\mathcal{B}_\infty$. Further, $\overline{B_k}$ maps under $g_{a,b}^{\circ m_k}$ onto $[l_2,r_2]$. Thus, we can choose $x_k\in\partial B_k\subset C_{a,b}$ such that $g_{a,b}^{\circ m_k}(x_k)=l_2$. As $\varphi_{a,b}$ is a non-decreasing semi-conjugacy, it follows that $\{x_k\}$ is an increasing sequence. 
We claim that $x_k \to l_1$. Indeed, the increasing sequence $\{x_k\}$ has a limit $x_\infty$, and by continuity of $\varphi_{a,b}$, we have $\varphi_{a,b}(x_\infty)=p_1$. Suppose that $x_\infty< l_1$. By Lemma~\ref{phi_preimage_lem}, $\varphi_{a,b}^{-1}(p_1)=[l_1,r_1]$, which contradicts the fact that $x_\infty\in\varphi_{a,b}^{-1}(p_1)$. Thus, $x_k \to l_1$.

To finish the proof, first note that the sequence $x_k$ eventually lies in $I_1$. Hence, $x_k\in I_1\cap C_{a,b}$ and $g_{a,b}^{\circ m_k}(x_k)=l_2\in I_2\cap C_{a,b}$ (for $k$ large). This proves that $g_{a,b}^{\circ m_k}(I_1\cap C_{a,b})\cap (I_2\cap C_{a,b})\neq\emptyset$, for $k$ large enough. Therefore, $g_{a,b}:C_{a,b}\to C_{a,b}$ is topologically transitive.
\end{proof}

\begin{remark}
Proposition~\ref{transitive_prop} can also be established by constructing a dense orbit in $C_{a,b}$ as the pull--back of a dense orbit for the doubling map on $\mathbb{S}^1$.
\end{remark}

We are now ready to establish Devaney chaos of $g_{a,b}$ on $C_{a,b}$.
\begin{prop}\label{chaos_inside_tongue}
    The map $g_{a,b}$ is Devaney chaotic restricted to $C_{a,b}$. Hence, $C_{a,b}$ is the unique maximal chaotic set for $g_{a,b}\vert_{\mathbb{S}^1}$.
\end{prop}
\begin{proof}
    Transitivity of $g_{a,b}$ on $C_{a,b}$ follows from Proposition~\ref{transitive_prop}. 

We will now argue that the periodic points of $g_{a,b}\vert_{C_{a,b}}$ are dense in $C_{a,b}$. To this end, first note that by transitivity of $g_{a,b}\vert_{C_{a,b}}$, the non-wandering set of $g_{a,b}:\mathbb{S}^1\to\mathbb{S}^1$ is the union of $C_{a,b}$ and the unique attracting cycle. By \cite[Chapter~III, Exercise~2.2, p.~214]{MS}, the closure of the periodic points of $g_{a,b}$ on $\mathbb{S}^1$ contains $C_{a,b}$. Since all but finitely many such periodic points lie in $C_{a,b}$, we conclude that the periodic points of $g_{a,b}$ in $C_{a,b}$ are dense in~$C_{a,b}$.
    
    That $C_{a,b}$ is the unique maximal chaotic set now follows from the above facts and Lemma~\ref{hyperbolic}.    
\end{proof}

\subsection{Conformal repeller and symbolic dynamics of the maximal chaotic set}\label{conf_rep_subsec}

In order to study the parameter dependence of the maximal chaotic set $C_{a,b}$ as $(a,b)$ runs over a tongue, we need to utilize the holomorphic extension of $g_{a,b}$. To this end, we will upgrade hyperbolicity of the one-dimensional dynamical system $g_{a,b}\vert_{C_{a,b}}$ to uniform expansion and topological exactness of a holomorphic dynamical system obtained by restricting $g_{a,b}$ to a complex neighborhood of $C_{a,b}$. To formulate this precisely, we need the notion of a \emph{conformal repeller} (cf. \cite[Chapter~5]{Z}, \cite[\S 16]{MRU}). 

\subsubsection{Conformal repeller}
\begin{defn}\label{defn_repeller}
    Let $h:V\to\C$ be a holomorphic map, where $V$ is an open subset of $\mathbb{C}$, and $\Lambda$ be a compact subset of $V$. Then the triplet $(\Lambda,V,h)$ is called a \emph{conformal repeller} if the following are satisfied.
    \begin{enumerate}[\upshape ({{CR}}-1)]
        \item\label{CR1} There exist $C>0, \alpha>1$, such that $|(h^{\circ n})'(z)| \geq C\alpha^n,\ \forall\ z \in \Lambda$ and $n \in \mathbb{N}$. 
        \item\label{CR2} $h^{-1}(V)$ is relatively compact in $V$ with $\Lambda=\{z \in V: h^{\circ n}(z)\in V,\ \forall\ n\in\N\}$.
        \item\label{CR3} For any open set $U$ with $U\cap \Lambda \neq \emptyset$, there exists $n\in\N$ so that $\Lambda \subset h^{\circ n}(U\cap \Lambda)$.
    \end{enumerate}
\end{defn}

The following lemma is an immediate consequence of Conditions~\ref{CR2} and~\ref{CR3} of the definition of a conformal repeller.

\begin{lem}\cite[Propositions~5.2,~5.3]{Z}\label{completely_invariant}
  Let $(\Lambda,V,h)$ be a conformal repeller. Then, $\Lambda$ is completely invariant under $h$. Moreover, if $\Lambda$ is not a singleton, then it is a perfect set.
\end{lem}

We now show that $g_{a,b}$ indeed gives rise to a conformal repeller in a neighborhood of $C_{a,b}$.

\begin{lem}\label{conformal_repeller}
    Let $(a,b)$ be a parameter in a tongue of the complexified DSM family. There exists an open neighborhood $V_{a,b}$ of $C_{a,b}$ such that $(C_{a,b},V_{a,b},g_{a,b})$ is a conformal repeller.
\end{lem}
\begin{proof}
As $C_{a,b}$ is a hyperbolic set for $g_{a,b}$ by Lemma~\ref{hyperbolic}, Property~\ref{CR1} is immediate. 

Thanks to the hyperbolicity of $g_{a,b}$ on $C_{a,b}$, the arguments of \cite[Chapter~V, Lemma~2.1]{CG} apply verbatim to the current situation to provide us with a conformal metric $\rho(z)|dz|$ in a neighborhood $W$ of $C_{a,b}$ that is expanded by the map $g_{a,b}$. In particular, for any compact set $K$ in $W$, there exists $\lambda>1$ such that $||Dg_{a,b}(z)||_\rho\geq \lambda$ for $z\in K$. We define $V_{a,b}:=\{z\in W: d_\rho(z, C_{a,b})<\epsilon\}$, where $\epsilon>0$ is chosen to be small enough to ensure that $d_\rho(f(z),C_{a,b})\geq\lambda_0 d_\rho(z,C_{a,b})$, for all $z\in\left(g_{a,b}\vert_{V_{a,b}}\right)^{-1}(V_{a,b})$ and some $\lambda_0>1$ independent of $z$. It now follows from the construction that $\left(g_{a,b}\vert_{V_{a,b}}\right)^{-1}(V_{a,b})$ is compactly contained in $V_{a,b}$ and the non-escaping set of $g_{a,b}:V_{a,b}\to\C$ is precisely $C_{a,b}$ (cf. \cite[Theorem~19.1]{M}, \cite[Theorem~6.4]{Z}). This proves that Property~\ref{CR2} is satisfied by $(C_{a,b},V_{a,b},g_{a,b})$.

Finally, to prove Property~\ref{CR3}, pick an open set $Y$ in $\C$ that intersects $C_{a,b}$ non-trivially. By Proposition~\ref{chaos_inside_tongue}, we can find a repelling $k-$periodic point of $g_{a,b}$ in $Y\cap C_{a,b}$. Possibly after shrinking $Y$ to a linearizing domain of this periodic point, we can assume that $g_{a,b}^{\circ k}(Y)\supset Y$. Hence, the sets $g_{a,b}^{\circ jk}(Y)$, $j\geq 0$, form an increasing family of open sets.
Since $Y$ intersects the Julia set $J_{a,b}$ (by Lemma~\ref{julia_intesect}), it follows that the sequence of holomorphic maps $\{g_{a.b}^{\circ jk}\vert_{Y}\}_{j\geq 0}$ is not a normal family. By Montel's theorem, $\bigcup_{j\geq 0}g_{a,b}^{\circ jk}(Y)$ can omit at most two values in the Riemann sphere $\widehat{\C}$ (cf. \cite[Theorem~4.10]{M}). But the asymptotic values $0,\infty$ are already omitted, and hence $\bigcup_{j\geq 0}g_{a,b}^{\circ jk}(Y)=\C^*\supset C_{a,b}$. 
By compactness of $C_{a,b}$ and the nesting of the open sets $g_{a,b}^{\circ jk}(Y)$, we can find $n\in\N$ such that $g_{a,b}^{\circ nk}(Y)\supset C_{a,b}$, and we are done.
\end{proof}

\begin{remark}
    An alternative proof of Property~\ref{CR2} can be given using the convergence of the critical orbits to the unique attracting cycle of $g_{a,b}$. The desired conformal metric that is expanded by $g_{a,b}$ is given by the Poincar{\'e} metric on a suitable neighborhood of $C_{a,b}$ such that the neighborhood avoids $0, \infty$, and the post-critical set of $g_{a,b}$ (cf. \cite[Theorem~19.1]{M}).
\end{remark}

\begin{cor}\label{cantor_cor}
    $C_{a,b}$ is a Cantor set.
\end{cor}
\begin{proof}
    The set $C_{a,b}$ is perfect by Lemma \ref{completely_invariant} and it is totally disconnected since it has measure $0$ by Lemma \ref{hyperbolic}. Hence, $C_{a,b}$ is a Cantor set. 
\end{proof}

\subsubsection{Coding, distortion estimates, and pressure}
We now recall some standard facts from thermodynamic formalism for conformal repellers that will be used in the proof of real-analyticity of the Hausdorff dimension of maximal chaotic sets for maps lying in a given tongue of the complexified DSM family. Let $(a,b)$ be a parameter in a tongue of the complexified DSM family. 

According to \cite[\S 5.1]{Z} (cf. \cite[Theorem~3.5.2]{PU}), the dynamical system $g_{a,b}:C_{a,b}\to C_{a,b}$ admits Markov partitions of arbitrarily small diameter.
Let $\{R_0,...,R_{k-1}\}$ be a Markov partition for the  conformal repeller $(C_{a,b},V_{a,b},g_{a,b})$. Let $A$ be the transition matrix associated with this Markov partition; i.e., $A=(a_{ij})$ with $a_{ij}=1$ if $g(R_i)\supset R_j$, and $a_{ij}=0$ otherwise. Let $\Sigma_A$ be the corresponding subshift of finite type, and $\sigma:\Sigma_A\to\Sigma_A$ be the shift map. Then, by \cite[Theorem~3.5.7]{PU}, the map $g_{a,b}: C_{a,b} \to C_{a,b}$ 
is semi-conjugate to $\sigma:\Sigma_A\to\Sigma_A$. Specifically, there exists a continuous surjection 
$$
p_{a,b}:\Sigma_A \to C_{a,b}
$$ 
such that $p_{a,b} \circ \sigma= g_{a,b} \circ p_{a,b}$.

An $n-$tuple $\{x_0,\cdots,x_{n-1}\}\in \{0,..., k-1\}^\N$ is called \emph{admissible} if $A_{x_i,x_{i+1}}=1$, for $i\in\{0,\cdots,n-1\}$. As usual, an admissible $n-$tuple $\{x_0,\cdots,x_{n-1}\}$ defines a cylinder set in $\Sigma_A$, and such a cylinder set projects under $p_{a,b}$ to a (rank $n$) cylinder set in $C_{a,b}$:
$$
[x_0,...,x_{n-1}]_{a,b}=\{z \in C_{a,b}: g_{a,b}^{\circ j}(z)\in R_{x_j}, j=0,...,n-1 \}.
$$
We will also need the \emph{geometric potential} 
$$
\psi_{a,b}: C_{a,b} \to \mathbb{R},\ \psi_{a,b}(z)=-\log|g_{a,b}'(z)|.
$$
The $n$-th Birkhoff sum of $\psi_{a,b}$ is defined as 
$$
(S_n\psi_{a,b})(z)=\psi_{a,b}(z)+...+\psi_{a,b}(g_{a,b}^{\circ (n-1)}(z))=-\ln\vert\left(g_{a,b}^{\circ n}\right)'(z)\vert.
$$
Thanks to the uniform expansion of $g_{a,b}\vert_{C_{a,b}}$, it is easily seen that the diameters of rank $n$ cylinders in $C_{a,b}$ decay exponentially in $n$, implying that the semi-conjugacy $p_{a,b}:\Sigma_{A}\to C_{a,b}$  is H{\"o}lder continuous, where $\Sigma_A$ is endowed with the usual ultra-metric (see \cite[Theorem~3.5.7]{PU}).

\begin{lem}\label{holder}
We define $\widetilde{\psi}_{a,b}:=\psi_{a,b}\circ p_{a,b}: \Sigma_A \to \mathbb{R}$.  For $t\in\R$, the functions $t\cdot \widetilde{\psi}_{a,b}:\Sigma_A\to\R$ are H{\"o}lder continuous with the same exponent $\delta>0$.
\end{lem}
\begin{proof}
Let $\delta>0$ be the H{\"o}lder exponent of $p_{a,b}$. Since $t\cdot\psi_{a,b}$ is smooth in a neighborhood of $C_{a,b}$, it follows that $t\cdot \widetilde{\psi}_{a,b}$ is also $\delta-$H{\"o}lder, for $t\in\R$.    
\end{proof}

Finally, we introduce the notion of topological pressure, which will be used to study the parameter dependence of the Hausdorff dimension of the maximal chaotic set $C_{a,b}$. For simplicity of exposition, we only define it for a sub-shift of finite type (see \cite[\S 2.2]{PU} for a more general definition). 

For a continuous function $\xi:\Sigma_A\to\R$ and a cylinder set $B=[x_0,...,x_{n-1}]\subset\Sigma_A$, we define $\xi_B=\sup \{\xi(x):\ x \in B\}$.

\begin{defn}\cite[\S 3.3]{Z}\label{pressure_def}
    Let $\xi : \Sigma_A \to \mathbb{R}$ be continuous. Then the quantity
    \[
    P(\sigma,\xi):= \lim_{n \to \infty} \dfrac{1}{n} \log\left( \sum_{B \in \mathcal{S}_n}e^{(S_n\xi)_B}\right)
    \]
    is called the \emph{topological pressure} of the potential $\xi$.
    Here, set $\mathcal{S}_n$ is the collection of all rank $n$ cylinders in $\Sigma_A$.
\end{defn}

The following result describes the regularity and the graph of the pressure function $\R\ni t\mapsto P(\sigma,t\cdot\widetilde{\psi}_{a,b})$, which is crucial for the computation of the Hausdorff dimension of $C_{a,b}$.
\begin{prop}\label{pressure_properties}
   The map $\R\ni t \mapsto P(t):=P(t\cdot\widetilde{\psi}_{a,b})$ is real-analytic. Further, it satisfies $P'(t)<0$ for all $t\in\R$, $\displaystyle\lim_{t\to-\infty}P(t)=+\infty$, $\displaystyle\lim_{t\to+\infty}P(t)=-\infty$, and $P(0)=h_{\mathrm{top}}(\sigma:\Sigma_A\to\Sigma_A)>0$ (where $h_{\mathrm{top}}$ stands for the topological entropy). 
   It follows that $t\mapsto P(t)$ has a unique positive zero.
\end{prop}
\begin{proof}
  The first statement is a consequence of real-analyticity of the pressure function on the Banach space of $\delta-$H{\"o}lder continuous functions (see \cite[Proposition~4.25]{Z}). The proofs of the remaining properties can be found in \cite[\S 16.3, Proposition~16.3.1]{MRU}.
\end{proof}

\subsection{Real-analyticity of Hausdorff dimension of maximal chaotic set}\label{real-anal_subsec}

The following result due to Bowen relates the Hausdorff dimension of the limit set of a conformal repeller to the pressure function.
\begin{prop}\cite[Theorem~5.12]{Z}, \cite[Theorem~16.3.2]{MRU}\label{bowen_formula}\\
 For a parameter $(a,b)$ in a tongue of the complexified DSM family, the Hausdorff dimension of $C_{a,b}$ is equal to the unique solution of the equation 
 $$
 P(t)=P(\sigma,t\cdot\widetilde{\psi}_{a,b})=0.
 $$
\end{prop}

We are now ready to prove the main theorem of this section. Fix a parameter $(a_0,b_0)$ in the interior of some tongue $T$. Recall from Lemma~\ref{qc_def_lemma} that for any $(a,b)\in \Int{T}$, there exists a unique quasiconformal conjugacy $\Phi_{a,b}$ between $g_{a_0,b_0}$ and $g_{a,b}$ fixing $0,\infty$. By design, $\Phi_{a_0,b_0}\equiv\mathrm{id}$, and the quasiconformal maps $\Phi_{a,b}$ depend real-analytically on $a, b$. Note that the semi-conjugacy $p_{a,b}$ between $\sigma\vert_{\Sigma_A}$ and $g_{a,b}\vert_{C_{a,b}}$ agrees with $\Phi_{a,b}\circ p_{a_0,b_0}$.

We denote the Hausdorff dimension of $C_{a,b}$ by $t_{a,b}$ and note that
\begin{equation}
    \begin{split}
        P(\sigma,t_{a,b}\cdot\widetilde{\psi}_{a,b})=P(\sigma,-t_{a,b}\cdot \ln\vert g_{a,b}'\circ p_{a,b}\vert)=0,\\
     \mathrm{i.e.,}\   P(\sigma,-t_{a,b}\cdot \ln\vert g_{a,b}'\circ\Phi_{a,b} \circ p_{a_0,b_0}\vert)=0.
    \end{split}
    \label{hd_eqn}
\end{equation}

\begin{theorem}\label{analytic_dimension}
    The map $\Int{T}\ni(a,b) \mapsto t_{a,b}=\dim_{H}(C_{a,b})$ is real-analytic.
\end{theorem}
\begin{proof}
The map $\Phi_{a,b}$ is H{\"o}lder continuous and its exponent of H{\"o}lder continuity is locally bounded on $\Int{T}$; indeed, the dilatation of $\Phi_{a,b}$ is locally bounded on $\Int{T}$ and the exponent of H{\"o}lder continuity of a quasiconformal map depends only on the dilatation. Hence, the function $-t\cdot\ln\vert g_{a,b}'\circ\Phi_{a,b} \circ p_{a_0,b_0}\vert:\Sigma_A\to\R$ is H{\"o}lder continuous and its exponent of H{\"o}lder continuity is locally bounded on $\Int{T}$.

The quasiconformal conjugacies $\Phi_{a,b}$ as well as the maps $g_{a,b}$ in the complexified DSM family depend real-analytically on $a,b$. One can show as in \cite{Rue82} that the map
$$
(t,a,b)\mapsto -t\cdot\ln\vert g_{a,b}'\circ\Phi_{a,b} \circ p_{a_0,b_0}\vert
$$
from a subset of $\R^3$ to the Banach space of H{\"o}lder continuous functions of a given exponent
is real-analytic. Thus, it follows from analytic dependence of pressure on H{\"o}lder spaces that the map $(t,a,b)\mapsto P(\sigma,-t\cdot\ln\vert g_{a,b}'\circ\Phi_{a,b} \circ p_{a_0,b_0}\vert)$ is real-analytic.

Thanks to Equation~\eqref{hd_eqn}, real-analytic dependence of $t_{a,b}$ on the parameters $a,b$ would follow by the Implicit Function Theorem if we show that $\dfrac{d}{dt} P(\sigma, t\cdot\widetilde{\psi}_{a,b})\neq 0$. According to \cite[Corollary~3]{Rue82} (cf. \cite[Theorem~16.4.10]{MRU}), the $t-$derivative is given by the formula
$$
\dfrac{d}{dt} P(\sigma, t\cdot\widetilde{\psi}_{a,b})=\int_{\Sigma_A} \widetilde{\psi}_{a,b}\ d\mu_{t,a,b},
$$
where $\mu_{t,a,b}$ is a $\sigma-$invariant, ergodic, probability measure on $\Sigma_A$ (more precisely, it is the unique $\sigma-$invariant Gibbs state for the potential $t\cdot\widetilde{\psi}_{a,b}$, see \cite[Proposition~1.14]{Bow} or \cite[Proposition~13.7.12]{MRU}). By the Birkhoff ergodic theorem, we have that 
$$
\displaystyle\int_{\Sigma_A} \widetilde{\psi}_{a,b}\ d\mu_{t,a,b}=\lim_n \dfrac{1}{n} (S_n\widetilde{\psi}_{a,b})(x)=-\lim_n \dfrac{1}{n}\ln\vert (g_{a,b}^{\circ n})'(p_{a,b}(x))\vert,
$$
for $\mu_{t,a,b}-$a.e. $x\in\Sigma_A$. By Lemma~\ref{conformal_repeller} and Property~\ref{CR1}, there exist $C>0$ and $\alpha>1$ such that $\vert(g_{a,b}^{\circ n})'\vert\geq C\alpha^n$, $n\geq 1$, on $C_{a,b}$. A simple calculation now shows that $\displaystyle\int_{\Sigma_A} \widetilde{\psi}_{a,b}\ d\mu_{t,a,b}\leq -\ln\alpha<0$, and we are done.
\end{proof}

\end{document}